\RequirePackage{fix-cm}
\documentclass[smallextended]{svjour3} 
\smartqed  
\usepackage{comment}
\usepackage[colorlinks]{hyperref}
\usepackage{graphicx}
\usepackage{amsmath,amssymb}
\usepackage{amscd}
\usepackage{makecell,multirow,diagbox}
\usepackage{mathrsfs}
\usepackage{stmaryrd}
\usepackage[caption=false]{subfig}
\usepackage{bm}
\usepackage[misc,geometry]{ifsym}

\DeclareMathOperator{\od}{\rm d}
\DeclareMathOperator{\opi}{\pi}

\numberwithin{theorem}{section}
\numberwithin{lemma}{section}
\numberwithin{equation}{section}
\numberwithin{remark}{section}

\begin{document}

\title{Nodal auxiliary space preconditioning for the surface de Rham complex}

\author{Yuwen Li
}


\institute{Yuwen Li(\Letter) \\
yuwenli925@gmail.com \at \newline
          Department of Mathematics, The Pennsylvania State University, \\University Park, PA 16802, USA \\   
          \newline
\emph{Current address}: School of Mathematical Sciences, Zhejiang University,\\ Hangzhou, Zhejiang 310058, China \\
\newline
Communicated by Doug Arnold 
}

\date{Received: date / Accepted: date}

\maketitle

\begin{abstract}
  This work develops optimal preconditioners for the discrete H(curl) and H(div) problems on two-dimensional surfaces by nodal auxiliary space preconditioning [R.~Hiptmair, J.~Xu: SIAM J.~Numer.~Anal.~\textbf{45}, 2483-2509 (2007)]. In particular, on unstructured triangulated surfaces, we develop fast and user-friendly preconditioners for the edge and face element discretizations of  curl-curl and grad-div problems based on inverting several discrete surface Laplacians. The proposed preconditioners lead to efficient iterative methods for computing harmonic tangential vector fields on discrete surfaces. Numerical experiments on two- and three-dimensional hypersurfaces are presented to test the performance of those surface preconditioners.
\end{abstract}

\section{Introduction}
Discretizations of partial differential equations (PDEs) typically yield sparse algebraic systems of linear equations with a huge number of unknowns. In order to achieve reasonable efficiency, those  large-scale discrete linear systems should be solved by fast linear solvers. In theory and practice, multilevel iterative solvers such as the geometric multigrid (cf.~\cite{Brandt1977,BankDupont1981,Hackbusch1985,Xu1992}) and algebraic multigrid (AMG) (cf.~\cite{BrandtMcCormickRuge1985,RugeStuben1987,VanekBrezinaMandel2001,BankSmith2002,XuZikatanov2017}) are the most efficient linear solvers for discretized PDEs on unstructured grids. Moreover, the convergence speed and robustness of these multilevel solvers could be improved when they are used in Krylov subspace methods, e.g., the preconditioned conjugate gradient (PCG) method, as preconditioners. On Euclidean domains, we refer to \cite{Hackbusch1985,BramblePasciakWangXu1991,XuZikatanov2002,XuZikatanov2017} for the classical theory of multilevel methods.

In recent decades, numerical methods for solving PDEs on surfaces has been a popular and important research area, see \cite{DeckelnickDziukElliott2005,DziukElliott2013} and references therein for an introduction. To efficiently implement numerical PDE schemes on surfaces, fast surface linear solvers are indispensable. We refer to e.g., \cite{Holst2001,AksoyluKhodakovskySchroder2005,KornhuberYserentant2008,BonitoPasciak2012,Li2021SISC} for specific fast Poisson-type solvers on surfaces. Besides nodal discretizations of elliptic PDEs, there have been many works devoted to numerical analysis of saddle-point systems of PDEs on surfaces, see e.g., \cite{HolstStern2012,CockburnDemlow2016,BonitoDemlowLicht2020} for the surface mixed Hodge Laplacian, mixed elliptic, and Stokes equations. Those numerical PDEs are built upon discrete divergence and curl and utilize the edge and face finite elements on discrete surfaces. Due to large kernels of curl and divergence operators, the algebraic systems resulting from discretized PDEs involving curl or div could not be efficiently solved by the standard AMG.

To the best of our knowledge, optimal iterative solvers for surface PDEs discretized by edge and face finite elements are still missing in the literature. In contrast, on Euclidean domains, with solid theoretical foundation, the discrete H(curl) and H(div)  systems could be efficiently solved by geometric multigrids \cite{Hiptmair1997,Hiptmair1999SINUM,ArnoldFalkWinther2000,Zikatanov2008,ChenNochettoXu2009,VassilevskiWang1992}, as well as Krylov subspace methods preconditioned by the popular Hiptmair--Xu (HX) preconditioner \cite{HiptmairXu2007}. The HX framework preconditions the inverse of discrete curl-curl and grad-div   elliptic operators using the inverse of several nodal element discrete Laplacians, which could be further approximated by well-established fast Poisson solvers. Without using a grid hierarchy, HX preconditioners are user-friendly and important building blocks of complex systems in real-world numerical simulations (cf.~\cite{Xu2010}).

For the edge element discretization of curl-curl problems and face element discretization of grad-div problems on surfaces, we develop optimal preconditioners by generalizing nodal auxiliary space HX preconditioning  \cite{HiptmairXu2007}. The theoretical analysis is based on a discrete stable decomposition of edge and face finite element spaces on surfaces. For example, the edge element space on a two-dimensional surface could be stably split as the sum of a high-frequency space and four surface nodal element spaces with the help of simple auxiliary transfer operators.  As a consequence, the corresponding curl-curl preconditioner makes use of the inverse of four discrete surface Laplacians on the 2-d surface. In contrast, the classical HX preconditioner on flat domains utilizes only three discrete inverse planar Laplacians in $\mathbb{R}^2$. Replacing the surface Laplacian with AMG cycles or parallel AMG preconditioners, the surface HX preconditioners could be independent of grid hierarchy and be easily implemented on any reasonable triangulated surfaces.  

In finite element exterior calculus, the Hodge--Laplace equation is an important model problem under extensive investigation 
  in recent years (cf.~\cite{ArnoldFalkWinther2006,ArnoldFalkWinther2010,HolstStern2012,Demlow2017,Li2019SINUM,HongLiXu2021}).
When solving the discrete Hodge Laplacian on domains with nontrivial topology, it is crucial to capture its kernel, the space of discrete harmonic forms or harmonic vector fields. In addition, the discrete harmonic space has many applications in computational geometry, electromagnetism and computer graphics, see, e.g., \cite{HiptmairOstrowski2002,FisherSchroderDesbrunHoppe2007,XuZhangCOXiong2009,RodriguezBertolazziGhiloniValli2013}. As far as we know, optimal solvers for computing harmonic fields have not been rigorously investigated in the existing literature. 
Our surface HX preconditioner yields a new optimal iterative method for computing 
harmonic tangential vector fields on discrete surfaces. In particular, a minimum residual (MINRES) method (cf.~\cite{ChoiPaigeSaunders2011,PaigeSaunders1975}) is used to find a basis of the kernel of the surface Hodge Laplacian in mixed form. With the help of a block diagonal surface HX preconditioner, the convergence speed of the MINRES iteration is shown to be uniform with respect to the grid size. 

\subsection{Notation}
In the rest of this section, we introduce the notation for abstract operator preconditioning.
For a Hilbert space $V$, let $(\bullet,\bullet)_V$ denote its inner product, $\|\bullet\|_V$ the $V$-norm, $V^\prime$ the dual space of $V,$ $[V]^\ell$ the Cartesian product of $\ell$ copies of $V,$ and $\langle\bullet,\bullet\rangle=\langle\bullet,\bullet\rangle_{V'\times V}$ the action of $V'$ on $V$. 
Given a linear operator $g: V_1\rightarrow V_2$, let $R(g)$ denote its range, $N(g)$ the kernel of $g$, and $g^\prime: V_2^\prime\rightarrow V_1^\prime$ the adjoint of $g$, i.e.,
\begin{equation*}
        \langle g^\prime r,v\rangle=\langle r,gv\rangle,\quad\forall v\in V_1,~\forall r\in V_2^\prime.
\end{equation*}
For a bounded linear operator $A: V\rightarrow V^\prime$, we say it is symmetric and positive-definite (SPD) provided $\forall v\in V$, $\langle Av,v\rangle\geq0$,  $\langle Av,v\rangle=0\Longrightarrow v=0$, and
\begin{align*}
        \langle Av_1,v_2\rangle=\langle Av_2,v_1\rangle,\quad\forall v_1, v_2\in V.
\end{align*} 
Similarly, we say a bounded linear operator $B: V^\prime\rightarrow V$ is SPD provided
$\forall r\in V^\prime$, $\langle r,Br\rangle\geq0$, $\langle r,Br\rangle=0\Longrightarrow r=0$, and
\begin{align*}
        \langle r_1,Br_2\rangle=\langle r_2,Br_1\rangle,\quad\forall r_1, r_2\in V^\prime.
\end{align*}
The SPD operators $A$ and $B$ define  inner products $V$ and $V^\prime$ by 
\begin{align*}
    &(v_1,v_2)_A:=\langle A v_1,v_2\rangle,\quad\forall v_1,~ v_2\in V,\\
    &(r_1,r_2)_B:=\langle r_1,Br_2\rangle,\quad\forall r_1,~ r_2\in V^\prime.
\end{align*}
Let $\|\bullet\|_A$ denote the norm in $V$ corresponding to $(\bullet,\bullet)_A$, and $\|\bullet\|_B$ the norm in $V^\prime$ associated with  $(\bullet,\bullet)_B$. Let  $$\kappa(BA):=\|BA\|_{V\rightarrow V}\|(BA)^{-1}\|_{V\rightarrow V}=\lambda_{\max}(BA)/\lambda_{\min}(BA)$$ be the operator conditioner number, where $\lambda_{\max}(BA)>0$, $\lambda_{\min}(BA)>0$ are the maximum and minimum eigenvalues of $BA$, respectively.
The following fictitious space lemma \cite{Nepomnyaschikh1992} is useful for estimating the condition number and thus developing uniform preconditioners, see, e.g., \cite{Xu1996,HiptmairXu2007}.
\begin{lemma}[Fictitious space lemma]\label{FSP}
Let $V$, $\bar{V}$ be Hilbert spaces and $A: V\rightarrow V^\prime$, $\bar{A}: \bar{V}\rightarrow\bar{V}^\prime$ be SPD operators.
Assume $\pi: \bar{V}\rightarrow V$ is a surjective linear operator, and 
\begin{itemize}
\item There exists a constant $c_1>0$ such that  $\|\pi\bar{v}\|_A\leq c_1\|\bar{v}\|_{\bar{A}}$ for each $\bar{v}\in\bar{V};$
    \item 
There exists a constant $c_2>0$ such that given any $v\in V,$ some $\bar{v}\in\bar{V}$ satisfies 
$$\pi\bar{v}=v,\quad\|\bar{v}\|_{\bar{A}}\leq c_2\|v\|_A.$$
\end{itemize}
Then for $B:=\pi\bar{A}^{-1}\pi^\prime: V^\prime\rightarrow V$ we have
\begin{equation*}
\begin{aligned}
c_1^{-2}\langle r,A^{-1}r\rangle&\leq\langle r,Br\rangle\leq c_2^2\langle r,A^{-1}r\rangle,\quad\forall r\in V^\prime,\\
    \kappa(BA)&\leq \left(c_1c_2\right)^2.
\end{aligned}
\end{equation*}
\end{lemma}

The rest of this paper is organized as follows. In Section \ref{secdeRham}, we introduce continuous and finite element de Rham complexes on surfaces. Section \ref{secInterpPiola} presents useful properties of interpolations on surfaces and Piola transformations between surfaces. In Section \ref{secHX}, we develop nodal auxiliary space preconditioners for the surface discrete H(curl) and H(div) problems. Section \ref{secharmonic} is devoted to fast computation of tangential harmonic vector fields by iterative methods. The proposed preconditioners are tested in several numerical experiments in Section \ref{secNE}.

\section{Surface de Rham complex}\label{secdeRham}
Let $\mathcal{M}$ be a smooth surface without boundary ($\partial\mathcal{M}=\emptyset$) in $\mathbb{R}^3$. Naturally $\mathcal{M}$ is endowed with a metric, which is the pullback of the Euclidean metric in $\mathbb{R}^3$ via the embedding $\mathcal{M}\hookrightarrow\mathbb{R}^3$. Let $\delta(x)$ be the signed distance function of $\mathcal{M}$ such that $|\delta(x)|$ is the distance from the point $x\in\mathbb{R}^3$ to $\mathcal{M}$,  $\delta(x)>0$ if $x$ is on the exterior side of $\mathcal{M}$ and $\delta(x)<0$ if $x$ is on the interior side. Then $\bm{\nu}:=\widetilde{\nabla}\delta$ is a smooth unit outward normal vector field on $\mathcal{M}$, where $\widetilde{\nabla}$ is the gradient operator in $\mathbb{R}^3.$ 

\subsection{Differential operators on surfaces}\label{seclinear}
Let $\mathcal{U}$ be a tubular neighborhood of $\mathcal{M}.$ We assume that $\mathcal{U}$ is sufficiently narrow such that $\delta(x)$, $\bm{\nu}(x)$ and the projection $a: \mathcal{U}\rightarrow\mathcal{M}$   
\begin{equation}\label{aprojection}
    a(x):=x-\delta(x)\bm{\nu}(x)\end{equation}
are well-defined at any point  $x\in\mathcal{U}$, see \cite{DemlowDziuk2007}.
A function $v$ on $\mathcal{M}$ could be extended in $\mathcal{U}$ as 
\[
v^\ell(x):=v(a(x)),\quad\forall x\in\mathcal{U}.
\]
Clearly $v^\ell$ is the constant extension of $v$ along $\bm{\nu}$, the normal direction of $\mathcal{M}.$ Let $\bm{v}$ be a tangential vector field along $\mathcal{M}$. The surface/tangential gradient, divergence, rotational gradient and curl along $\mathcal{M}$ are given by
\begin{equation}\label{dM}
\begin{aligned}
    &\nabla v=\nabla_\mathcal{M}v:=\widetilde{\nabla} v^\ell-(\bm{\nu}\cdot\widetilde{\nabla} v^\ell)\bm{\nu},\\
    &\nabla\cdot \bm{v}=\nabla_\mathcal{M}\cdot \bm{v}:=\widetilde{\nabla}\cdot \bm{v}^\ell-\bm{\nu}\cdot(\widetilde{\nabla}\bm{v}^\ell)\bm{\nu},\\
    &\nabla^\perp v=\nabla_\mathcal{M}^\perp v:=(\nabla_\mathcal{M} v)\times\bm{\nu},\\
    &\nabla\times \bm{v}=\nabla_\mathcal{M}\times \bm{v}:=\nabla_\mathcal{M}\cdot(\bm{v}\times\bm{\nu}),
\end{aligned}
\end{equation}
respectively.
In fact, $\nabla_\mathcal{M}\cdot$ is the $L^2(\mathcal{M})$-adjoint of $-\nabla_\mathcal{M}$ and 
\begin{align*}
    &(\nabla_\mathcal{M}\times)\circ\nabla_\mathcal{M}=0,\quad (\nabla_\mathcal{M}\cdot)\circ\nabla_\mathcal{M}^\perp=0.
\end{align*}  The composite  $\Delta_\mathcal{M}=(\nabla_\mathcal{M}\cdot)\circ\nabla_\mathcal{M}$ is the Laplace--Beltrami operator (surface Laplacian) on $\mathcal{M}$. 
On a surface $\mathcal{M}_\alpha$, we adopt the notation 
\[ \text{d}_\alpha^-=\nabla_{\mathcal{M}_\alpha}^\perp,~\text{d}_\alpha=\nabla_{\mathcal{M}_\alpha}\cdot\quad\text{ or }\quad\text{d}_\alpha^-=\nabla_{\mathcal{M}_\alpha},~\text{d}_\alpha=\nabla_{\mathcal{M}_\alpha}\times\]
such that $\text{d}_\alpha\circ\text{d}_\alpha^-=0$, where the subscript $\alpha$ might be \emph{suppressed} or $\alpha$=1, 2 or $h$ later. It is noted that $\od_\alpha^-$, $\od_\alpha$ are defined a.e. if $\mathcal{M}_\alpha$ is piecewise smooth. 

By $L^2(T\mathcal{M})$ we denote  the space of $L^2$ tangential vector fields on $\mathcal{M}$, where $T\mathcal{M}$ is the set of tangential fields on $\mathcal{M}.$ Consider the following spaces 
\begin{align*}
    H(\od^-)&:=\big\{v\in L^2(\mathcal{M}): \od^-v\in L^2(T\mathcal{M})\big\},\\
    H(\od)&:=\big\{\bm{v}\in L^2(T\mathcal{M}): \od\bm{v}\in L^2(\mathcal{M})\big\}.
\end{align*}
Here $H^1(\mathcal{M})=H(\od^-)$ and we have
the surface de Rham complex
\begin{equation}\label{2ddeRham}
    \begin{CD}
    H(\od^-)@>\od^->>H(\od)@>\od>>L^2(\mathcal{M}).
    \end{CD}
\end{equation}

Let $(\bullet,\bullet)_{\mathcal{M}}$ be the $L^2(\mathcal{M})$-inner product $(\bullet,\bullet)_{\mathcal{M}}$. For a constant $c>0$, our model variational problem is to find $\bm{u}\in H(\text{d})$ such that
\begin{equation}\label{Vd}
		(\text{d}\bm{u},\text{d}\bm{v})_\mathcal{M}+c(\bm{u},\bm{v})_\mathcal{M}=(\bm{g},\bm{v})_{\mathcal{M}},\quad     \bm{v}\in H(\text{d}),
\end{equation}
where $\bm{g}\in L^2(T\mathcal{M})$. We assume that $c>c_0>0$ for some fixed $c_0$ such that the nearly singular case is excluded.

Let  $\varphi: \mathcal{M}_1\rightarrow\mathcal{M}_2$ be a diffeomorphism between two manifolds $\mathcal{M}_1$ and $\mathcal{M}_2$. Given a scalar-valued function $v$ on $\mathcal{M}_1,$ the tangent map $\mathcal{D}\varphi: L^2(T\mathcal{M}_1)\rightarrow L^2(T\mathcal{M}_2)$ satisfies
\begin{equation}\label{gradM1M2}
    \nabla_{\mathcal{M}_1}v|_x=(\mathcal{D}\varphi)^*\nabla_{\mathcal{M}_2}(v\circ\varphi^{-1})|_{\varphi(x)},\quad\forall x\in\mathcal{M}_1.
\end{equation}
Here $(\mathcal{D}\varphi)^*: L^2(T\mathcal{M}_2)\rightarrow L^2(T\mathcal{M}_1)$ is the adjoint linear mapping of $\mathcal{D}\varphi$.

Let $d\sigma_i$ be the surface measure on $\mathcal{M}_i$, and $d\sigma_2(\varphi(x))=\mu(x)d\sigma_1(x)$. We summarize Piola transforms on surfaces (cf.~\cite{CockburnDemlow2016,Monk2003}) as follows. 
\begin{subequations}\label{PiolaM12}
\begin{align}
    &\mathcal{P}^{\nabla}_\varphi: L^2(\mathcal{M}_1)\rightarrow L^2(\mathcal{M}_2)\qquad\mathcal{P}^\nabla_\varphi v=\mathcal{P}^{\nabla^\perp}_\varphi v:=v\circ\varphi^{-1},\\
    &\mathcal{P}^{\nabla\cdot}_\varphi: L^2(T\mathcal{M}_1)\rightarrow L^2(T\mathcal{M}_2)\qquad\mathcal{P}^{\nabla\cdot}_\varphi \bm{v}:=\frac{1}{\mu}(\mathcal{D}\varphi)\bm{v},\\
    &\mathcal{P}^{\nabla\times}_\varphi: L^2(T\mathcal{M}_1)\rightarrow L^2(T\mathcal{M}_2)\qquad\mathcal{P}^{\nabla\times}_\varphi \bm{v}:=\big[(\mathcal{D}\varphi)^{-1}\big]^*\bm{v}.
\end{align}
\end{subequations}
Similarly to the Euclidean case, it holds that
\begin{subequations}
\begin{align}
    \mathcal{P}^{\text{d}}_{\varphi^{-1}}&=(\mathcal{P}^{\text{d}}_\varphi)^{-1},\label{Piolainverse}\\
    \mathcal{P}^{\text{d}}_\varphi\circ\text{d}^-_1&=\text{d}_2^-\circ\mathcal{P}^{\od^-}_\varphi.\label{Piolacommuting}
\end{align}
\end{subequations}

\subsection{Finite element discretization}\label{subsecDLB}
When devising numerical schemes for solving \eqref{Vd}, we assume that $\mathcal{M}$ is approximated by a polyhedral surface  $\mathcal{M}_h$ with triangular faces, and  $\mathcal{M}_h$ is sufficiently close to $\mathcal{M}$ such that $\mathcal{M}_h\subset\mathcal{U}.$ Let   $\mathcal{T}_h$ denote the collection of all 2-d faces of $\mathcal{M}_h,$ and $\mathcal{F}_h$ the set of $1$-d faces/edges in $\mathcal{T}_h$.  
\begin{figure}[tbhp]
\centering
\includegraphics[width=10cm,height=5.5cm]{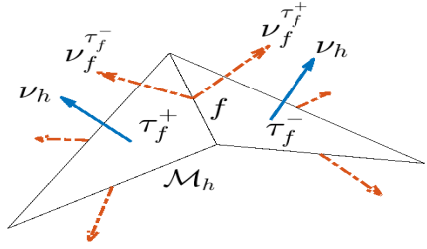}
\caption{Unit normals (solid vectors) and conormals (dashed vectors) on $\mathcal{M}_h$.}
\label{normalfigure}
\end{figure}

Let $\bm{\nu}_h$ be the piecewise constant unit normal vector field on $\mathcal{M}_h$ such that $\bm{\nu}\cdot\bm{\nu}_h>0$. 
Let $\bm{\nu}_f^\tau$ be the outward unit \emph{conormal} vector for the face $f\in\mathcal{F}_h$ of an element $\tau\in\mathcal{T}_h$. In other words, $\bm{\nu}_f^\tau$ is orthogonal to both $\bm{\nu}_h|_\tau$ and $f$ and is pointing towards the exterior of $\tau$. For each  $f\in\mathcal{F}_h$, $\bm{t}_f$ is a unit tangent vector along $f$, and $\tau_f^+, \tau_f^-\in\mathcal{T}_h$ are the two elements sharing $f$, see Figure \ref{normalfigure}.

By $\mathbb{P}_r(U)$ we denote the space of polynomials of degree at most $r$ on a flat surface $U$ is denoted. Let $\hat{\tau}$ be a reference triangle in $\mathbb{R}^2$, $\hat{\bm{x}}$ the coordinate position vector in $\mathbb{R}^2$, and $\varphi_\tau: \hat{\tau}\rightarrow\tau$ be an affine mapping such that $\varphi_\tau(\hat{\tau})=\tau$. We consider the following local shape function spaces
\begin{equation*}
\begin{aligned}
  \mathbb{P}^{\nabla\times}(\tau)&:=\mathcal{P}^{\nabla\times}_{\varphi_\tau}\big([\mathbb{P}_0(\hat{\tau})]^2+\big\{\hat{\bm{p}}\in\big[\mathbb{P}_0(\hat{\tau})\big]^2: \hat{\bm{p}}\cdot\hat{\bm{x}}=0\big\}\big),\\
  \mathbb{P}^{\nabla\cdot}(\tau)&:=\mathcal{P}^{\nabla\cdot}_{\varphi_\tau}\big([\mathbb{P}_0(\hat{\tau})]^2+\mathbb{P}_0(\hat{\tau})\hat{\bm{x}}\big).
\end{aligned}
\end{equation*}
Let  $\nabla_h=\nabla_{\mathcal{M}_h}$, $\nabla_h^\perp=\nabla^\perp_{\mathcal{M}_h}$,  $\nabla_h\cdot=\nabla_{\mathcal{M}_h}\cdot$,  $\nabla_h\times=\nabla_{\mathcal{M}_h}\times$.
The surface nodal element space is
\[
    H_h(\nabla_h)=H_h(\nabla_h^\perp):=\big\{v_h\in C^0(\mathcal{M}_h): v_h|_\tau\in\mathbb{P}_1(\tau)~\forall \tau\in\mathcal{T}_h\big\}.\]
Lowest-order edge ($N_0$ \cite{Nedelec1980,Monk2003}) and face ($RT_0$ \cite{RaviartThomas1977}) element spaces on $\mathcal{M}_h$ are
\begin{equation}\label{Hcurldef}
    \begin{aligned}
    H_h(\nabla_h\times):=&\big\{\bm{v}_h\in L^2(T\mathcal{M}_h): \bm{v}_h|_\tau\in \mathbb{P}^{\nabla\times}(\tau)\text{ for each }\tau\in\mathcal{T}_h,\\
    &\qquad\llbracket \bm{v}_h\rrbracket_{t,f}=0\text{ for each }f\in\mathcal{F}_h\big\},
\end{aligned}
\end{equation}
\begin{equation}\label{Hdivdef}
   \begin{aligned}
    H_h(\nabla_h\cdot):=&\big\{\bm{v}_h\in L^2(T\mathcal{M}_h): \bm{v}_h|_\tau\in \mathbb{P}^{\nabla\cdot}(\tau)\text{ for each }\tau\in\mathcal{T}_h,\\
    &\qquad\llbracket \bm{v}_h\rrbracket_{\nu,f}=0\text{ for each }f\in\mathcal{F}_h\big\},
\end{aligned} 
\end{equation}
respectively. Here tangential and conormal jumps in \eqref{Hcurldef}, \eqref{Hdivdef} are
\begin{align*}
    &\llbracket \bm{v}_h\rrbracket_{t,f}:=\bm{v}_h|_{\tau_f^+}-(\bm{v}_h|_{\tau_f^+}\cdot\bm{\nu}_f^{\tau_f^+})\bm{\nu}_f^{\tau_f^+}-\big[\bm{v}_h|_{\tau_f^-}-(\bm{v}_h|_{\tau_f^-}\cdot\bm{\nu}_f^{\tau_f^-})\bm{\nu}_f^{\tau_f^-}\big],\\
    &\llbracket \bm{v}_h\rrbracket_{\nu,f}:=\bm{v}_h|_{\tau_f^+}\cdot\bm{\nu}_f^{\tau_f^+}+\bm{v}_h|_{\tau_f^-}\cdot\bm{\nu}_f^{\tau_f^-}.
\end{align*}

Let $(\bullet,\bullet)_h=(\bullet,\bullet)_{\mathcal{M}_h}$ be the $L^2$ inner product on $\mathcal{M}_h$.  
The finite element discretization of \eqref{Vd} seeks $\bm{u}_h\in H_h(\text{d}_h)$ such that
\begin{equation}\label{Vhd}
		(\text{d}_h\bm{u}_h,\text{d}_h\bm{v}_h)_h+c(\bm{u}_h,\bm{v}_h)_h=(\bm{g}_h,\bm{v}_h)_h,\quad     \bm{v}_h\in H_h(\text{d}_h),
\end{equation}
where $\bm{g}_h\in L^2(T\mathcal{M}_h)$ approximates $\bm{g}$. The bilinear form in \eqref{Vhd} induces a linear operator $A^\text{d}_h: H_h(\text{d}_h)\rightarrow H_h(\text{d}_h)^\prime$ by
\begin{equation*}
    \langle A^\text{d}_h\bm{v}_h,\bm{w}_h\rangle:=(\text{d}_h\bm{v}_h,\text{d}_h\bm{w}_h)_h+c(\bm{v}_h,\bm{w}_h)_h,\quad\forall\bm{v}_h,\bm{w}_h\in H_h(\text{d}_h).
\end{equation*}
We shall develop efficient preconditioners for the discrete operator $A^\text{d}_h.$ 

Let $N(\od)^\perp$ (resp.~$N(\od_h)^\perp$) be the $L^2$-orthogonal complement of $N(\od)$ (resp.~$N(\od_h)^\perp$) in $H(\od)$ (resp.~$H_h(\od_h)$).
We present the continuous and discrete Poincar\'e inequalities (cf.~\cite{ArnoldFalkWinther2006,HolstStern2012}).
\begin{lemma}[Poincar\'e inequality]
Let $\|\bullet\|=\|\bullet\|_{L^2(\mathcal{M}_h)}$. There exists constants $c_P>0$, $c_{h,P}>0$ such that 
\begin{subequations}
\begin{align}
    &\|\bm{v}\|_{L^2(\mathcal{M})}\leq c_P\|{\od}\bm{v}\|_{L^2(\mathcal{M})},\quad\forall\bm{v}\in N(\od)^\perp,\label{Poincare}\\
    &\|\bm{v}_h\|\leq c_{h,P}\|{\od_h}\bm{v}_h\|,\quad\forall\bm{v}_h\in N(\od_h)^\perp.\label{discretePoincare}
\end{align}
\end{subequations}
\end{lemma}
It is possible that $c_{h,P}$ depends on the grid size $h$ of $\mathcal{M}_h$. The work \cite{HolstStern2012} shows that $c_{h,P}$ is an absolute constant provided there exists a uniformly bounded cochain projection on $\mathcal{M}_h.$ 

For $f\in\mathcal{F}_h$, $\tau\in\mathcal{T}_h$, let $h_f=\text{diam}(f)$, $h_\tau=\text{diam}(\tau)$, and $h$ be the piecewise constant on $\mathcal{M}_h$ such that $h|_\tau=h_\tau$. Let $h_{\max}:=\max_{\tau\in\mathcal{T}_h}h_\tau$, $h_{\min}:=\min_{\tau\in\mathcal{T}_h}h_\tau$ and $c_{\rm qu}:=h_{\max}/h_{\min}$.
By $C_1\lesssim C_2$ we mean $C_1\leq CC_2$ with $C$ being a generic constant dependent only on $\mathcal{M}$, the local mesh quality of $\mathcal{T}_h$, and $r$, $c_{h,P}$, $c_0$, $c_{\rm qu}$. We say $C_1\simeq C_2$ provided $C_1\lesssim C_2$ and $C_2\lesssim C_1$. For SPD operators $A$ and $B$, $A\lesssim B$ provided $(v,v)_A\lesssim(v,v)_B$ for all $v.$ 

Let $|\bullet|$ denote the Euclidean norm. We make the following common assumption in surface finite element literature  (cf.~\cite{Demlow2009,DednerMadhavenStinner2013,CockburnDemlow2016})
\begin{equation}\label{nunuh}
|\delta(x)|\lesssim h_\tau,\quad|\bm{\nu}(x)-\bm{\nu}_h(x)|\lesssim h_\tau,\quad\forall\tau\in\mathcal{T}_h,~\forall x\in\tau.
\end{equation}
Given $f\in\mathcal{F}_h$, it follows that $\bm{\nu}_h|_{\tau_f^+}-\bm{\nu}_h|_{\tau_f^-}=O(h_f)$ and 
\begin{equation}\label{nuplusminus}
    |\bm{\nu}_f^{\tau_f^+}+\bm{\nu}_f^{\tau_f^-}|\lesssim h_f.
\end{equation}

\section{Interpolations and Piola transformations on surfaces}\label{secInterpPiola}
As a first step, we must present several estimates for interpolations and Piola transformations on surfaces.

\subsection{Surface interpolations}
Let $d\sigma$ (resp.~$d\sigma_h$) denote the surface measure of $\mathcal{M}$ (resp.~$\mathcal{M}_h$), and $\Omega_z$ the union of elements in $\mathcal{T}_h$ sharing the grid vertex $z$. The Cl\'ement interpolation $\mathcal{I}_h$ in the nodal element space $H_h(\nabla_h)$  is given by 
\begin{equation*}
    (\mathcal{I}_hv)(z):=\frac{1}{\int_{\Omega_z}1d\sigma_h}\int_{\Omega_z}vd\sigma_h\quad\text{ for each vertex $z$ in $\mathcal{T}_h$}.
\end{equation*}
On each element $\tau\in\mathcal{T}_h,$
let  $\pi_\tau^{\text{d}}$ be the canonical interpolation onto $\mathbb{P}^{\text{d}}(\tau)$. By $H^1(TU), W^{k,p}(TU)$ we denote spaces of tangential  vector fields with standard Sobolev regularity on a surface $U$. For $\bm{v}\in H^1(T\tau)$,   $\pi_\tau^{\nabla\times}\bm{v}\in\mathbb{P}^{\nabla\times}(\tau)$ and $\opi_\tau^{\nabla\cdot}\bm{v}\in\mathbb{P}^{\nabla\cdot}(\tau)$ are determined by
\begin{align}
    \int_f(\pi_\tau^{\nabla\times}\bm{v})\cdot \bm{t}_fds&=\int_f\bm{v}\cdot\bm{t}_fds,\\
    \int_f(\pi_\tau^{\nabla\cdot}\bm{v})\cdot\bm{\nu}^\tau_fqds&=\int_f\bm{v}\cdot \bm{\nu}^\tau_fqds,\quad\forall f\subset\partial\tau,~f\in\mathcal{F}_h.\label{Pidivtau}
\end{align}
Let $\|\bullet\|_\tau=\|\bullet\|_{L^2(\tau)}$, and  $\Omega_\tau$ be the union of elements in $\mathcal{T}_h$ sharing a vertex with $\tau\in\mathcal{T}_h.$ For $v\in H^1(\mathcal{M}_h)$, $\bm{v}\in H^1(T\tau)$, classical results (cf.~\cite{DemlowDziuk2007,CockburnDemlow2016}) yield
\begin{subequations}\label{Clement}
\begin{align}
\|\mathcal{I}_hv\|_\tau&\lesssim\|v\|_{\Omega_\tau},\label{Ihl2}\\
h_\tau^{-1}\|v-\mathcal{I}_hv\|_\tau+|\mathcal{I}_hv|_{H^1(\tau)}&\lesssim |v|_{H^1(\Omega_\tau)},\label{Ihh1}\\
\|\bm{v}-\opi^{\od}_\tau \bm{v}\|_\tau&\lesssim h_\tau|\bm{v}|_{H^1(\tau)}.\label{Pidiv}
\end{align}
\end{subequations}

\begin{figure}[tbhp]
\centering
\includegraphics[width=8cm,height=5cm]{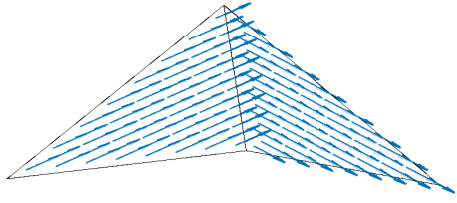}
\caption{A tangential vector field  on $\mathcal{M}_h$.}
\label{tangentialfigure}
\end{figure}
Let  $\pi_h^\text{d}$ be the canonical interpolation onto $H_h(\text{d}_h)$ such that 
$(\pi_h^{\text{d}}\bm{v})|_\tau=\pi_\tau^\text{d}(\bm{v}|_\tau)$, $\forall\tau\in\mathcal{T}_h.$ The space of tangential and piecewise $H^1$-fields on $\mathcal{M}_h$ is
\begin{align*}
    H^1_h(T\mathcal{M}_h)&:=\big\{\bm{v}_h\in L^2(T\mathcal{M}_h): \bm{v}_h|_\tau\in H^1(T\tau)~\forall\tau\in\mathcal{T}_h\big\},
\end{align*}
see Figure \ref{tangentialfigure} for example. 
The next lemma presents sufficient conditions for the well-posedness of $\pi_h^{\od},$ which follows from the definitions \eqref{Hcurldef} and \eqref{Hdivdef} and a trace theorem.
\begin{lemma}\label{pihd}
For $\bm{v}\in H_h^1(T\mathcal{M}_h)$, there exists a unique $\pi_h^{\nabla\cdot}\bm{v}\in H_h(\nabla_h\cdot)$ if $\llbracket \bm{v}\rrbracket_{\nu,f}=0$ across each $f\in\mathcal{F}_h$; and  $\pi_h^{\nabla\times}\bm{v}\in H_h(\nabla_h\times)$ is well-defined provided $\llbracket \bm{v}\rrbracket_{t,f}=0$ for each $f\in\mathcal{F}_h$.
\end{lemma}

Let $\opi_h^{\text{d}^-}$ be the linear nodal interpolation onto $H_h(\nabla_h).$ It holds that
\begin{equation}\label{Pihcommute}
    \text{d}^-_h\circ\opi_h^{\text{d}^-}=\opi_h^{\text{d}}\circ\od^-_h.
\end{equation}
Given a non-tangential $H^1$-field $\bm{v}$ on $\mathcal{M}_h,$ let $\bm{v}^\parallel$ (resp.~$\bm{v}^\perp$) denote its  tangential (resp.~normal) component, see Figure \ref{decompfigure}.
We note that $\pi_h^{\nabla\times}\bm{v}=\pi_h^{\nabla\times}(\bm{v}^\parallel)$ is  well-defined. 
However, $\opi_h^{\nabla\cdot}: [H^1(\mathcal{M}_h)]^3\rightarrow H_h(\nabla_h\cdot)$ becomes ambiguous because $\bm{v}_h\in [H^1(\mathcal{M}_h)]^3$ has a discontinuous conormal component  across each face $f\in\mathcal{F}_h$ and $\llbracket \bm{v}\rrbracket_{\nu,f}$ does not vanish. To remedy this situation, we propose a modified $H(\rm div)$ interpolation $\bar{\opi}_\tau^{\nabla\cdot}\bm{v}$ for each $\tau\in\mathcal{T}_h$ and $\bm{v}\in [H^1(\tau)]^3$ by
\begin{equation}\label{Pibardivtau}
    \int_f(\bar{\opi}_\tau^{\nabla\cdot}\bm{v})\cdot\bm{\nu}^\tau_fds=\delta_\tau^f\int_f\bm{v}\cdot \bm{\nu}^{\tau^+_f}_fds,~\forall f\subset\partial\tau,~f\in\mathcal{F}_h,
\end{equation}
where $\delta_\tau^f=1$ if $\tau=\tau_f^+$ and $\delta_\tau^f=-1$ otherwise. Note that \eqref{Pidivtau} uses the outward conormal element-wise while \eqref{Pibardivtau} depends on a pre-assigned conormal $\bm{\nu}^{\tau_f^+}_f$ for each face $f$. When the surface $\mathcal{M}_h$ is globally flat, we have $\bm{\nu}_f^{\tau_f^+}=-\bm{\nu}_f^{\tau_f^-}$ and the two interpolations $\pi_h^{\nabla\cdot}$ and $\bar{\pi}_h^{\nabla\cdot}$ coincide. Let $\bar{\pi}_h^{\nabla\cdot}\bm{v}\in H_h(\nabla_h\cdot)$ be the global  interpolant such that \[
(\bar{\pi}_h^{\nabla\cdot}\bm{v})|_\tau=\bar{\pi}_\tau^{\nabla\cdot}(\bm{v}|_\tau),\quad\forall\tau\in\mathcal{T}_h.
\]
\begin{figure}[tbhp]
\centering
\includegraphics[width=12cm,height=5.5cm]{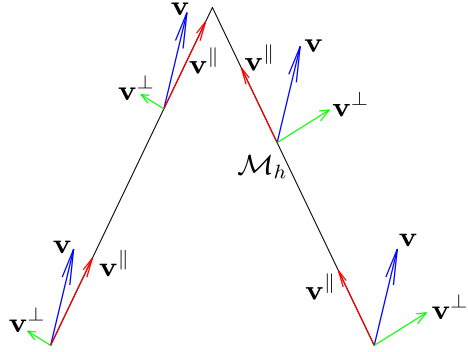}
\caption{Decomposition of a non-tangential vector field $\bm{v}$.}
\label{decompfigure}
\end{figure}
The trace theorem implies that the domain of $\bar{\pi}_h^{\nabla\cdot}$ contains $[H^1(\mathcal{M}_h)]^3$. Moreover, $\bar{\pi}_h^{\nabla\cdot}\bm{v}$ is also well-defined for any discontinuous $\bm{v}\in H_h^1(T\mathcal{M}_h)$.
For convenience, we may use the trivial notation $\bar{\pi}_h^{\nabla\times}=\pi_h^{\nabla\times}$. The properties of $\bar{\pi}^{\nabla\times}_h$ and $\bar{\pi}^{\nabla\cdot}_h$ are presented in the next lemma. 
\begin{lemma}\label{Pidinterp}
Let $\tau\in\mathcal{T}_h$ and $\bm{v}\in [H^1(\tau)]^3$. It holds that
\begin{subequations}
\begin{align}
\|\bm{v}-\bar{\pi}_h^{\nabla\cdot}\bm{v}\|_\tau&\lesssim \|\bm{v}^\perp\|_\tau+h_\tau\|\bm{v}\|_{H^1(\tau)},\label{divinterp}\\
\|\bm{v}-\pi_h^{\nabla\times}\bm{v}\|_\tau&\lesssim \|\bm{v}^\perp\|_\tau+h_\tau|\bm{v}^\parallel|_{H^1(\tau)}.\label{curlinterp}
\end{align}
\end{subequations}
For $\bm{v}_h\in[\mathbb{P}_1(\tau)]^3$, it holds that
\begin{subequations}\label{Pibound}
\begin{align}
\|\bar{\pi}_h^{\od}\bm{v}_h\|_\tau&\lesssim\|\bm{v}_h\|_\tau,\label{Pibound1}\\
\|{\od}_h\bar{\pi}_h^{\od}\bm{v}_h\|_\tau&\lesssim\|\bm{v}_h\|_\tau+\|{\od}_h\bm{v}_h^\parallel\|_\tau.    \label{Pibound2}
\end{align}
\end{subequations}
\end{lemma}
\begin{proof}
Using $\bm{v}=\bm{v}^\perp+\bm{v}^\parallel$, $\pi_h^{\nabla\cdot}\bm{v}=\pi_h^{\nabla\cdot}(\bm{v}^\parallel)$, and the triangle inequality, we have
\begin{equation}\label{triangle}
    \|\bm{v}-\bar{\pi}_h^{\nabla\cdot}\bm{v}\|_\tau\leq \|\bm{v}^\perp\|_\tau+\|\bm{v}^\parallel-\pi_h^{\nabla\cdot}\bm{v}^\parallel\|_\tau+\|\pi_h^{\nabla\cdot}\bm{v}-\bar{\pi}_h^{\nabla\cdot}\bm{v}\|_\tau.
\end{equation}
For each 1-d face $f\subset\partial\tau$,
it follows from the trace inequality
\begin{equation*}
    \|\bm{v}\|_{L^1(f)}\lesssim\|\bm{v}\|_\tau+h_\tau\|\nabla_\tau\bm{v}\|_\tau
\end{equation*}
and $\bm{\nu}_f^{\tau^+_f}+\bm{\nu}_f^{\tau^-_f}=O(h_f)$ on $f$ that
\begin{equation}\label{PiPibar}
\|\pi_h^{\nabla\cdot}\bm{v}-\bar{\pi}_h^{\nabla\cdot}\bm{v}\|_\tau\lesssim h_\tau\|\bm{v}\|_{\tau}+h^2_\tau|\bm{v}|_{H^1(\tau)}.
\end{equation}
Therefore combining \eqref{triangle} with \eqref{Pidiv} and \eqref{PiPibar} leads to \eqref{divinterp}.

Similarly, using $\pi_h^{\nabla\times}\bm{v}=\pi_h^{\nabla\times}(\bm{v}^\parallel)$, the triangle inequality
\begin{equation*}
    \|\bm{v}-\pi_h^{\nabla\times}\bm{v}\|_\tau\leq \|\bm{v}^\perp\|_\tau+\|\bm{v}^\parallel-\pi_h^{\nabla\times}\bm{v}^\parallel\|_\tau,
\end{equation*}
and \eqref{Pidiv}, we obtain \eqref{curlinterp}. The bound \eqref{Pibound1} follows from a scaling argument. Using $\pi_h^{\od}\bm{v}_h=\pi_h^{\od}\bm{v}^\parallel_h$, an inverse inequality and \eqref{PiPibar}, we have
\begin{equation*}
\begin{aligned}
\|{\od}_h\bar{\pi}_h^{\od}\bm{v}_h\|_\tau&\leq\|{\od}_h\pi_h^{\od}\bm{v}_h\|_\tau+\|{\od}_h(\bar{\pi}_h^{\od}-\pi_h^\text{d})\bm{v}_h\|_\tau\\
&\lesssim\|{\od}_h\bm{v}^\parallel_h\|_\tau+h_\tau^{-1}\|(\bar{\pi}_h^{\od}-\pi_h^\text{d})\bm{v}_h\|_\tau\\
&\lesssim\|{\od}_h\bm{v}_h^\parallel\|_\tau+\|\bm{v}_h\|_\tau
\end{aligned}
\end{equation*}
and verify \eqref{Pibound2}. The proof is complete.
\qed\end{proof}

\subsection{Surface Piola transformations}
In the following, we describe the Piola transformation between the smooth surface $\mathcal{M}$ and the discrete surface $\mathcal{M}_h.$ Define
\begin{align*}
    H:=\nabla\bm{\nu},\quad P:=I-\bm{\nu}\otimes\bm{\nu},\quad P_h:=I-\bm{\nu}_h\otimes\bm{\nu}_h,
\end{align*}
where $I$ is the identity mapping. The restriction $a|_{\mathcal{M}_h}: \mathcal{M}_h\rightarrow\mathcal{M}$ of the projection $a: \mathcal{U}\rightarrow\mathcal{M}$ is bijective. With slight abuse of notation, we simply denote $a=a|_{\mathcal{M}_h}$ such that the inverse $a^{-1}: \mathcal{M}\rightarrow\mathcal{M}_h$ exists.

Let $x\in\mathcal{M}_h$ and $\mu_h$ be the density function on $\mathcal{M}_h$ such that $d\sigma(a(x))=\mu_h(x)d\sigma_h(x)$. It is shown in \cite{DemlowDziuk2007} that  surface gradients are related as 
\begin{equation}\label{gradMMh}
\begin{aligned}
\nabla_{\mathcal{M}_h}w_h(x)&=P_h(x)[I-\delta H](x)P(a(x))\nabla_{\mathcal{M}}\tilde{w}_h(a(x)),\\
\nabla_{\mathcal{M}}\tilde{w}_h(a(x))&=[I-\delta H]^{-1}(x)\left(I-\frac{\bm{\nu}_h\otimes\bm{\nu}}{\bm{\nu}_h\cdot\bm{\nu}}\right)\nabla_{\mathcal{M}_h}w_h(x),
\end{aligned}
\end{equation}
where $w_h$ is a function on $\mathcal{M}_h,$ and 
$\tilde{w}_h:=w_h\circ a^{-1}$ is the lifting  on $\mathcal{M}$.
Let $\bm{v}_h$ and $\bm{v}$ be  tangential vector fields on $\mathcal{M}_h$ and $\mathcal{M},$ respectively. Using \eqref{gradM1M2},  \eqref{PiolaM12}, \eqref{Piolainverse}, \eqref{gradMMh}, we obtain the following surface Piola  transformation
\begin{subequations}
\begin{align}
    (\mathcal{P}^{\nabla\times}_a\bm{v}_h)(a(x))&=[I-\delta H]^{-1}(x)\left(I-\frac{\bm{\nu}_h\otimes\bm{\nu}}{\bm{\nu}_h\cdot\bm{\nu}}\right)\bm{v}_h(x),\label{Pacurl}\\
    (\mathcal{P}^{\nabla\times}_{a^{-1}} \bm{v})(x)&=P_h(x)[I-\delta H](x)\bm{v}(a(x)),\label{Painversecurl}\\
    (\mathcal{P}^{\nabla\cdot}_a \bm{v}_h)(a(x))&=\frac{1}{\mu_h(x)}P(a(x))[I-\delta H](x)\bm{v}_h(x),\\
     (\mathcal{P}^{\nabla\cdot}_{a^{-1}} \bm{v})(x)&=\mu_h(x)\left(I-\frac{\bm{\nu}\otimes\bm{\nu}_h}{\bm{\nu}_h\cdot\bm{\nu}}\right)[I-\delta H]^{-1}(x)\bm{v}(a(x)),\label{Painversediv}
\end{align}
\end{subequations}
see Figure \ref{Piolafigure} for the illustration.
\begin{figure}[tbhp]
\centering
\includegraphics[width=9cm,height=5.5cm]{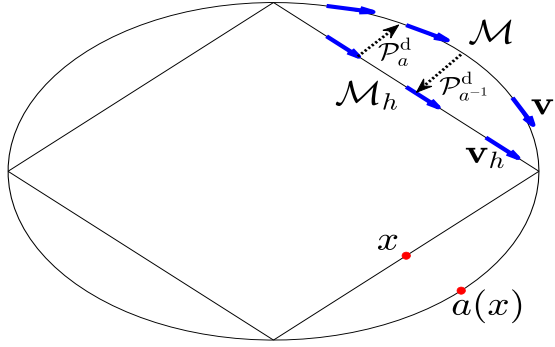}
\caption{Piola transformation between $\mathcal{M}$ and $\mathcal{M}_h$.}
\label{Piolafigure}
\end{figure}
We next present properties of Piola transforms.
\begin{lemma}\label{Piolaproperty}
Let $\bm{v}\in L^2(T\mathcal{M})$ be a tangential vector field on $\mathcal{M}$. We have
\begin{equation}\label{L2bound}
    |\mathcal{P}^{\od}_{a^{-1}}\bm{v}(x)-\bm{v}^\ell(x)|\lesssim h_\tau|\bm{v}(a(x))|,\quad\forall x\in\tau,~\forall\tau\in\mathcal{T}_h.
\end{equation}
In addition, for $\tau\in\mathcal{T}_h$, $\bm{v}_h\in L^2(T\tau)$, $\bm{v}\in H^1(Ta(\tau))$, $x\in\tau,$ we have
\begin{subequations}\label{L2H1Pbound}
\begin{align}
|\mathcal{P}^{\od}_{a}\bm{v}_h(a(x))|&\lesssim|\bm{v}_h(x)|,\quad|\mathcal{P}^{\od}_{a^{-1}}\bm{v}(x)|\lesssim|\bm{v}(a(x))|,\label{L2Pbound}\\
|\nabla_\tau\mathcal{P}^{\od}_{a^{-1}}\bm{v}(x)|&\lesssim|\bm{v}(a(x))|+ |\nabla\bm{v}(a(x))|,\label{H1Pbound}\\
|\nabla\mathcal{P}^{\od}_a\bm{v}_h(a(x))|&\lesssim|\bm{v}_h(x)|+ |\nabla_h\bm{v}_h(x)|.
\end{align}
\end{subequations}
\end{lemma}
\begin{proof}
In the case that $\text{d}=\nabla\cdot$, \eqref{L2bound} is derived in \cite{CockburnDemlow2016}. Given $\tau\in\mathcal{T}_h$ and a point $x\in\tau,$ using the formula \eqref{Painversecurl} and that $\bm{v}^\ell(x)\perp\bm{\nu}(x)$, we have  
\begin{equation}\label{tempcurl}
\begin{aligned}
&(\mathcal{P}^{\nabla\times}_{a^{-1}}\bm{v})(x)-\bm{v}^\ell(x)=P_h(x)[I-\delta H](x)\bm{v}(a(x))-\bm{v}(a(x))\\
&=\big[P_h(x)-\delta(x)P_h(x)H(x)-P(x)\big]\bm{v}(a(x))\\
&=\big[\bm{\nu}(x)\otimes\bm{\nu}(x)-\bm{\nu}_h(x)\otimes\bm{\nu}_h(x)-\delta(x)P_h(x)H(x)\big]\bm{v}(a(x)).
\end{aligned}
\end{equation}
Combining \eqref{tempcurl} with \eqref{nunuh} yields
\begin{equation}
    |(\mathcal{P}^{\nabla\times}_{a^{-1}}\bm{v})(x)-\bm{v}^\ell(x)|\lesssim h_\tau|\bm{v}(a(x))|.
\end{equation}
The estimates \eqref{L2H1Pbound} follows from the definitions \eqref{Pacurl}--\eqref{Painversediv}.
\qed\end{proof}

For $f\in\mathcal{F}_h,$ along the interface $a(f)\subset\mathcal{M}$ let $\llbracket \bullet\rrbracket_{t,a(f)}$, $\llbracket \bullet\rrbracket_{\nu,a(f)}$ be the tangential and conormal jumps  defined in the same fashion as $\llbracket \bullet\rrbracket_{t,f}$, $\llbracket \bullet\rrbracket_{\nu,f}$, respectively. The Piola transformation preserves tangential and conormal continuity of vector fields across interfaces. Given $\bm{v}_h\in H_h^1(T\mathcal{M}_h)$, we have
\begin{align*}
 &\llbracket \bm{v}_h\rrbracket_{t,f}=0\Longrightarrow\llbracket\mathcal{P}^{\nabla\times}_a\bm{v}_h\rrbracket_{t,a(f)}=0,\quad\llbracket \bm{v}_h\rrbracket_{\nu,f}=0\Longrightarrow\llbracket\mathcal{P}^{\nabla\cdot}_a\bm{v}_h\rrbracket_{\nu,a(f)}=0.
\end{align*}
Combining this fact and Lemma \ref{Piolaproperty} leads to
\begin{equation}\label{Pad}
    \mathcal{P}^{\rm d}_a(H_h(\od_h))\subset H(\od).
\end{equation}
Similarly, 
for $\bm{v}\in H^1(T\mathcal{M})$ and $f\in\mathcal{F}_h$ we have 
\begin{equation}\label{Padinv}
\begin{aligned}
&\mathcal{P}^{\od}_{a^{-1}}\bm{v}\in H_h^1(T\mathcal{M}_h),\\
&\llbracket\mathcal{P}^{\nabla\times}_{a^{-1}}\bm{v}\rrbracket_{t,f}=0,\quad\llbracket\mathcal{P}^{\nabla\cdot}_{a^{-1}}\bm{v}\rrbracket_{\nu,f}=0.
\end{aligned}
\end{equation}

\section{Preconditioning in H(curl) and H(div) on surfaces}\label{secHX}
In this section, we develop efficient nodal auxiliary space preconditioners for the discrete operator $A_h^{\text{d}}$, where the auxiliary space builds upon the surface nodal element space $H_h(\nabla_h)$, equipped with the inner product 
\begin{equation*}
    \langle A_h^\nabla v_h,w_h\rangle=(v_h,w_h)_{A_h^\nabla}=(\nabla_hv_h,\nabla_hw_h)_h+c(v_h,w_h)_h.
\end{equation*}
On a surface $\mathcal{M}_\alpha$ with $\alpha$ being suppressed or $\alpha=h,$ let $$\|\bullet\|_{H(\od_\alpha)}^2:=\|\bullet\|_{L^2(\mathcal{M}_\alpha)}^2+\|\od_\alpha\bullet\|_{L^2(\mathcal{M}_\alpha)}^2.$$ 
We consider the space of harmonic tangential vector fields on $\mathcal{M}$
\begin{equation*}
    \mathcal{H}(\text{d}):=\big\{\bm{u}\in L^2(T\mathcal{M}): \text{d}\bm{u}=0,~(\text{d}^-)^*\bm{u}=0\big\},
\end{equation*}
where $\text{d}^*$ is the $L^2(\mathcal{M})$-adjoint of $\text{d}$. 
The next lemma deals with the Hodge decomposition of vector fields on smooth surfaces.
\begin{lemma}[Hodge decomposition]\label{Hodge}
For any $\bm{v}\in H(\od)$, there exist $\bm{\phi}\in H^1(T\mathcal{M})$ and $p\in H(\od^-)$ such that
\begin{equation}\label{vsplittingeqn}
    \begin{aligned}
    &\bm{v}=\bm{\phi}+{\od}^-p,\\
    &\|\bm{\phi}\|_{L^2(\mathcal{M})}+\|p\|_{H(\od^-)}\lesssim\|\bm{v}\|_{L^2(\mathcal{M})},\\
    &\|\bm{\phi}\|_{H^1(\mathcal{M})}\lesssim\|\bm{v}\|_{L^2(\mathcal{M})}+\|{\od}\bm{v}\|_{L^2(\mathcal{M})}.
\end{aligned}
\end{equation}
In addition, we have $p\in C^0(\mathcal{M})\text{ when }\bm{v}\in L^s(T\mathcal{M})$ with $s>2.$
\end{lemma}
\begin{proof}
The $L^2(\mathcal{M})$-orthogonal Hodge decomposition (cf.~\cite{Schwarz1995,ArnoldFalkWinther2006}) of $\bm{v}$ reads
\begin{equation}\label{L2decomp}
    \bm{v}=\bm{q}+{\od^*}r+{\od^-}p,
\end{equation}
where  $\bm{q}\in\mathcal{H}(\od)$,  $r$ is in the domain of $\text{d}^*$, and $p\in N(\od^-)^\perp\subset H(\od^-)$ with
\begin{equation}\label{pl2}
    \|p\|_{L^2(\mathcal{M})}\leq c_P\|\od^-p\|_{L^2(\mathcal{M})}.
\end{equation}
Let $\bm{\phi}:=\bm{q}+{\od^*}r.$
The $L^2$-orthogonality of \eqref{L2decomp} implies
\begin{equation}\label{phil2}
    \|\text{d}^-p\|_{L^2(\mathcal{M})}+\|\bm{\phi}\|_{L^2(\mathcal{M})}\leq2\|\bm{v}\|_{L^2(\mathcal{M})}.
\end{equation}
Using $(\text{d}^-)^*\circ\text{d}^*=0,$ $\text{d}\bm{\phi}=\text{d}\bm{v}$ and the Gaffney inequality (cf.~\cite{Schwarz1995,ArnoldFalkWinther2006,Gaffney1951}) on $\mathcal{M}$, we have $\bm{\phi}\in H^1(T\mathcal{M})$ and
\begin{equation}\label{Gaffney}
\begin{aligned}
\|\bm{\phi}\|_{H^1(\mathcal{M})}&\lesssim\|\bm{\phi}\|_{L^2(\mathcal{M})}+\|\text{d}\bm{\phi}\|_{L^2(\mathcal{M})}+\|(\text{d}^-)^*\bm{\phi}\|_{L^2(\mathcal{M})}\\
&=\|\bm{\phi}\|_{L^2(\mathcal{M})}+\|\text{d}\bm{v}\|_{L^2(\mathcal{M})}.
\end{aligned}
\end{equation}
Collecting \eqref{L2decomp}--\eqref{Gaffney} finishes the proof of \eqref{vsplittingeqn}.

Due to the Sobolev embedding, it holds that $\bm{\phi}\in H^1(T\mathcal{M})\hookrightarrow L^s(T\mathcal{M})$ and $p\in H(\od^-)=H^1(\mathcal{M})\hookrightarrow L^s(\mathcal{M})$ with $s>2$. Hence we have $\od^-p=\bm{v}-\bm{\phi}\in L^s(T\mathcal{M})$ and thus $p\in W^{1,s}(\mathcal{M})\hookrightarrow C^0(\mathcal{M}).$ 
\qed\end{proof}
On a non-smooth polygonal surface, Hodge decompositions of vector fields could be found in \cite{BuffaCiarlet2001}.

In a Euclidean space, the classical HX preconditioner \cite{HiptmairXu2007} utilizes the space of globally continuous and piecewise linear vector fields. In differential geometry, Euclidean vector fields are intrinsically generalized as tangential vector fields on a smooth manifold. However, any vector field  tangential to a triangulated surface $\mathcal{M}_h$ cannot be continuous, see Figure \ref{tangentialfigure}. Therefore we relax the tangential condition and make use of the space $[H_h(\nabla_h)]^3$ of all continuous and non-tangential piecewise linear vector fields  on $\mathcal{M}_h$ as an auxiliary space. 
Now we are in a position to prove the main result for preconditioning.
\begin{theorem}\label{vhsplitting}
For any $\bm{v}_h\in H_h(\od_h),$ there exist $\tilde{\bm{v}}_h\in H_h(\od_h),$ $\bm{\phi}_h\in[H_h(\nabla_h)]^3$ and $p_h\in H_h(\od^-_h)$ such that
\begin{equation*}
\begin{aligned}
    &\bm{v}_h=\tilde{\bm{v}}_h+\bar{\pi}^{\od}_h\bm{\phi}_h+\od_h^-p_h,\\
    &\|h^{-1}\tilde{\bm{v}}_h\|^2+\|\bm{\phi}_h\|^2_{A_h^\nabla}+c\|p_h\|^2_{H(\od^-_h)}\lesssim \|\bm{v}_h\|^2_{A_h^{\od}}.
\end{aligned}
\end{equation*}
\end{theorem}
\begin{proof}
Given $\bm{v}_h\in H_h(\text{d}_h),$ we have $\mathcal{P}^\text{d}_a\bm{v}_h\in H(\text{d})$ by \eqref{Pad}. The Hodge decomposition of $\mathcal{P}_a^{\text{d}}\bm{v}_h\in H(\text{d})$ on the smooth $\mathcal{M}$ in Lemma \ref{Hodge} reads
\begin{equation}\label{Padvh}
    \mathcal{P}^\text{d}_a\bm{v}_h=\bm{\phi}+\text{d}^- p,
\end{equation}
where $\bm{\phi}\in H^1(T\mathcal{M})$, $p\in C^0(\mathcal{M}).$ It follows from  \eqref{vsplittingeqn}, \eqref{L2H1Pbound} and \eqref{Piolacommuting} that
\begin{equation}\label{phipbound}
\begin{aligned}
&\|\bm{\phi}\|_{L^2(\mathcal{M})}+\|{\rm d}^-p\|_{L^2(\mathcal{M})}\lesssim\|\mathcal{P}^{\rm d}_a\bm{v}_h\|_{L^2(\mathcal{M})}\lesssim\|\bm{v}_h\|,\\
&\|\bm{\phi}\|_{H^1(\mathcal{M})}\lesssim\|\mathcal{P}^\text{d}_a\bm{v}_h\|_{H({\rm d})}\lesssim\|\bm{v}_h\|+\|{\rm d}_h\bm{v}_h\|.
\end{aligned}
\end{equation}
Applying $\mathcal{P}^{\text{d}}_{a^{-1}}$ to \eqref{Padvh} and using the commuting property \eqref{Piolacommuting}, we obtain
\begin{equation}\label{vhdecomp1}
    \bm{v}_h=\mathcal{P}^{\rm d}_{a^{-1}}\bm{\phi}+{\rm d}_h^- \mathcal{P}^{{\rm d}^-}_{a^{-1}}p.
\end{equation}

It follows from the
property \eqref{Padinv} with $\bm{v}=\bm{\phi}$ and Lemma \ref{pihd} that $\pi^{\rm d}_h\mathcal{P}^{\rm d}_{a^{-1}}\bm{\phi}$ is well-defined. In addition,  $\mathcal{P}^{{\od}^-}_{a^{-1}}p=p\circ a^{-1}$ is continuous and the nodal interpolant $\pi^{{\od}^-}_h\mathcal{P}^{{\od}^-}_{a^{-1}}p$ exists.  
Now applying the canonical interpolation $\pi_h^{\rm{d}}$ to \eqref{vhdecomp1} and using  \eqref{Pihcommute} and ${\pi}^{\rm d}_h\mathcal{P}^{\rm d}_{a^{-1}}\bm{\phi}=\bar{\pi}^{\rm d}_h\mathcal{P}^{\rm d}_{a^{-1}}\bm{\phi}$, we have
\begin{equation}\label{vhdecomp}
\begin{aligned}
    &\bm{v}_h=\pi^{\rm d}_h\mathcal{P}^{\rm d}_{a^{-1}}\bm{\phi}+\pi_h^{{\rm d}}{\rm d}_h^-\mathcal{P}^{{\rm d}^-}_{a^{-1}}p\\
    &=\bar{\pi}^{\rm d}_h(\mathcal{P}^{\rm d}_{a^{-1}}\bm{\phi}-\mathcal{I}_h\bm{\phi}^\ell)+\bar{\pi}^\text{d}_h(\mathcal{I}_h\bm{\phi}^\ell)+{\rm d}_h^-{\pi}_h^{{\rm d}^-}\mathcal{P}^{{\rm d}^-}_{a^{-1}}p.
\end{aligned}
\end{equation}
There exists $p_h\in N({\rm d}^-_h)^\perp$ such that
\begin{equation}\label{ph}
    {\rm d}^-_hp_h={\rm d}_h^-{\pi}_h^{\rm d^-}\mathcal{P}^{{\rm d}^-}_{a^{-1}}p={\pi}_h^{\rm d}{\rm d}_h^-\mathcal{P}^{{\rm d}^-}_{a^{-1}}p.
\end{equation}
Then using the discrete Poincar\'e inequality \eqref{discretePoincare}, $\text{d}_h^- \mathcal{P}^{\text{d}^-}_{a^{-1}}p=\bm{v}_h-\mathcal{P}^\text{d}_{a^{-1}}\bm{\phi}$, $(I-\pi_h^{\text{d}})\bm{v}_h=0,$ and \eqref{Piolacommuting},   we have
\begin{equation*}
\begin{aligned}
\|p_h\|_{H(\od_h^-)}&\lesssim\|\pi_h^{\text{d}} \text{d}_h^-\mathcal{P}^{\text{d}^-}_{a^{-1}}p\|\leq\|(I-\pi_h^{\text{d}})\text{d}_h^-\mathcal{P}^{\text{d}^-}_{a^{-1}}p\|+\|\text{d}_h^-\mathcal{P}^{\text{d}^-}_{a^{-1}}p\|\\
&=\|(I-\pi_h^{\text{d}})\mathcal{P}^{\text{d}}_{a^{-1}}\bm{\phi}\|+\|\mathcal{P}^{\text{d}}_{a^{-1}}\text{d}^-p\|.
\end{aligned}
\end{equation*}
It then follows from the above estimate and \eqref{Pidiv},   \eqref{L2H1Pbound}, \eqref{phipbound} that
\begin{equation}\label{phbound}
\begin{aligned}
\|p_h\|_{H(\od_h^-)}&\lesssim\left(\sum_{\tau\in\mathcal{T}_h}h^2_\tau|{\mathcal{P}}^{\text{d}}_{a^{-1}}\bm{\phi}|_{H^1(\tau)}^2\right)^\frac{1}{2}+\|\bm{v}_h\|\\
&\lesssim h_{\max}\big(\|\bm{v}_h\|+\|\text{d}_h\bm{v}_h\|\big)+\|\bm{v}_h\|\lesssim\|\bm{v}_h\|.
\end{aligned}
\end{equation}
The other two components in the decomposition of $v_h$ are set to be  
\begin{align*}
    \tilde{\bm{v}}_h=\bar{\pi}^\text{d}_h({\mathcal{ P}}^\text{d}_{a^{-1}}\bm{\phi}-\mathcal{I}_h\bm{\phi}^\ell),\quad\bm{\phi}_h=\mathcal{I}_h\bm{\phi}^\ell.
\end{align*}
It is easy to see that for $\tau\in\mathcal{T}_h,$
\begin{equation}\label{phiellphi}
    \|\bm{\phi}^\ell\|_{\tau}\simeq\|\bm{\phi}\|_{L^2(a(\tau))},\quad\|\bm{\phi}^\ell\|_{H^1(\tau)}\simeq\|\bm{\phi}\|_{H^1(a(\tau))}.
\end{equation}
The following estimate is a consequence of \eqref{Clement}, \eqref{phipbound} and \eqref{phiellphi} 
\begin{equation*}
    \|\bm{\phi}_h\|\lesssim\|\bm{v}_h\|,\quad|\bm{\phi}_h|_{H^1(\mathcal{M}_h)}\lesssim\|\bm{v}_h\|+\|\text{d}_h\bm{v}_h\|.
\end{equation*}
As a result, we obtain
\begin{equation}\label{phihbound}
    \|\bm{\phi}_h\|_{A_h^\nabla}\lesssim\|\bm{v}_h\|_{A_h^\text{d}}.
\end{equation}
On each $\tau\in\mathcal{T}_h$, it follows from Lemma \ref{Pidinterp} that
\begin{equation}\label{vhtilde1}
\begin{aligned}
    \|\tilde{\bm{v}}_h\|_\tau&\leq\|(I-\bar{\pi}^\text{d}_h)(\mathcal{P}^\text{d}_{a^{-1}}\bm{\phi}-\mathcal{I}_h\bm{\phi}^\ell)\|_\tau+\|\mathcal{P}^\text{d}_{a^{-1}}\bm{\phi}-\mathcal{I}_h\bm{\phi}^\ell\|_\tau\\
    &\lesssim\|(\mathcal{P}^\text{d}_{a^{-1}}\bm{\phi}-\mathcal{I}_h\bm{\phi}^\ell)^\perp\|_\tau+h_\tau|\mathcal{P}^\text{d}_{a^{-1}}\bm{\phi}-\mathcal{I}_h\bm{\phi}^\ell|_{H^1(\tau)}\\
    &\qquad+\|\mathcal{P}^\text{d}_{a^{-1}}\bm{\phi}-\mathcal{I}_h\bm{\phi}^\ell\|_\tau\\
    &\lesssim\|\mathcal{P}^\text{d}_{a^{-1}}\bm{\phi}-\mathcal{I}_h\bm{\phi}^\ell\|_\tau+h_\tau|\mathcal{P}^\text{d}_{a^{-1}}\bm{\phi}-\mathcal{I}_h\bm{\phi}^\ell|_{H^1(\tau)}.
\end{aligned}
\end{equation}
Using \eqref{Ihh1} and \eqref{L2bound} and a triangle inequality, we obtain
\begin{equation}\label{term2}
    \begin{aligned}
    &\|\mathcal{P}^\text{d}_{a^{-1}}\bm{\phi}-\mathcal{I}_h\bm{\phi}^\ell\|_\tau\\
    &\leq\|\mathcal{P}^\text{d}_{a^{-1}}\bm{\phi}-\bm{\phi}^\ell\|_\tau+\|\bm{\phi}^\ell-\mathcal{I}_h\bm{\phi}^\ell\|_\tau\lesssim h_\tau\|\bm{\phi}^\ell\|_{H^1(\Omega_\tau)}.
    \end{aligned}
\end{equation}
Using \eqref{Ihh1}, \eqref{H1Pbound}, \eqref{phipbound}, we obtain
\begin{equation}\label{term3}
    |\mathcal{P}^\text{d}_{a^{-1}}\bm{\phi}-\mathcal{I}_h\bm{\phi}^\ell|_{H^1(\tau)}\lesssim\|\bm{\phi}^\ell\|_{H^1(\Omega_\tau)}.
\end{equation}
Combining \eqref{vhtilde1}--\eqref{term3} and \eqref{phiellphi}, \eqref{phipbound} yields
\begin{equation}\label{vhtildebound}
    \|h^{-1}\tilde{\bm{v}}_h\|\lesssim\left(\sum_{\tau\in\mathcal{T}_h}\|\bm{\phi}^\ell\|^2_{H^1(\Omega_\tau)}\right)^\frac{1}{2}\simeq\|\bm{\phi}\|_{H^1(\mathcal{M})}\lesssim\|\bm{v}_h\|+\|\text{d}_h\bm{v}_h\|.
\end{equation}
Finally we complete the proof with \eqref{vhdecomp}, \eqref{phbound}, \eqref{phihbound}, \eqref{vhtildebound}.
\qed\end{proof}

Let $S^\text{d}_h: H_h(\text{d}_h)^\prime\rightarrow H_h(\text{d}_h)$ be an SPD operator such that $$\|\bm{v}_h\|_{(S^\text{d}_h)^{-1}}\simeq c^\frac{1}{2}\|\bm{v}_h\|+\|h^{-1}\bm{v}_h\|,\quad\forall\bm{v}_h\in H_h(\text{d}_h).$$
In the multigrid literature, 
$S^\text{d}_h$ is known as a smoother, which could be any classical local relaxation such as the Jacobi and symmetrized Gauss--Seidel iteration.
In the following, we simply set $S^{\od}_h$ to be the operator corresponding to the inverse diagonal of the matrix for $A^{\text{d}}_h$, i.e., the Jacobi iteration. Recall the transfer operators  $\bar{\opi}_h^\text{d}: \big[H_h(\nabla_h)\big]^3\rightarrow H_h(\text{d}_h)$ and $\text{d}_h^-: H_h(\text{d}_h^-)\rightarrow H_h(\text{d}_h).$ We define the preconditioner $B_h^{\text{d}}: H_h(\text{d}_h)^\prime\rightarrow H_h(\text{d}_h)$ for $A_h^{\text{d}}$ as
\[
B_h^{\text{d}}:=S^{\text{d}}_h+\bar{\opi}_h^\text{d}\big[\big(A_h^\nabla\big)^{-1}\big]^3(\bar{\pi}_h^\text{d})^\prime+c^{-1}\text{d}_h^-\big(A_h^{\text{d}^-}\big)^{-1}(\text{d}_h^-)^\prime.
\]
Using \eqref{Pibound} and Theorem \ref{vhsplitting}, it is straightforward to verify the assumptions in Lemma \ref{FSP} with \begin{align*}
    &V=H_h(\text{d}_h),\quad A=A_h^\text{d},\\
    &\bar{V}=H_h(\text{d}_h)\times\big[H_h(\nabla_h)\big]^3\times H_h(\text{d}^-_h),\\
    &\bar{A}=(S^\text{d}_h)^{-1}\times \big[A_h^\nabla\big]^3\times c A_h^{\text{d}^-},\quad\opi=(I,\bar{\opi}_h^\text{d},\text{d}_h^-).
\end{align*} 
As a consequence of Lemma \ref{FSP}, we then obtain the following spectral equivalence and the condition number estimate 
\begin{equation}\label{kappabound}
    (A_h^{\text{d}})^{-1}\lesssim B_h^{\text{d}}\lesssim(A_h^{\text{d}})^{-1},\quad\kappa(B_h^{\text{d}}A_h^{\text{d}})\lesssim 1.
\end{equation}
In particular, when $\text{d}=\nabla\times$ or $\text{d}=\nabla\cdot$, we have
\begin{subequations}\label{Bhcurldiv}
\begin{align}
  B_h^{\nabla\times}:=&S^{\nabla\times}_h+\opi_h^{\nabla\times}\big[(A_h^\nabla)^{-1}\big]^3(\pi_h^{\nabla\times})^\prime+c^{-1}\nabla_h\big(A_h^\nabla\big)^{-1}\nabla_h^\prime,\label{Bhcurl}\\
  B_h^{\nabla\cdot}:=&S^{\nabla\cdot}_h+\bar{\pi}_h^{\nabla\cdot}\big[(A_h^\nabla)^{-1}\big]^3(\bar{\pi}_h^{\nabla\cdot})^\prime+c^{-1}\nabla_h^\perp\big(A_h^\nabla\big)^{-1}\nabla_h^{\perp\prime}.\label{Bhdiv2}
  \end{align}
\end{subequations}

Let $\{\bm{\phi}_j\}_{j=1}^J$ be a finite element basis of $H_h(\text{d}_h)$ and $\{\bm{\phi}_i^\prime\}_{i=1}^J$ the dual basis of $H_h(\text{d}_h)^\prime$ such that $\langle\bm{\phi}_i^\prime,\bm{\phi}_j\rangle=\delta_{ij}$.  Under the basis $\{\bm{\phi}_j\}_{j=1}^J$, let $\widetilde{A}^{\text{d}}_h$ (resp.~$\widetilde{B}^{\text{d}}_h$) denote the matrix representing $A^{\text{d}}_h$ (resp.~$B^{\text{d}}_h$),  $\widetilde{D}^{\text{d}}_h$ the diagonal of $\widetilde{A}^{\text{d}}_h$, $\widetilde{P}_h^{\text{d}}$ the matrix for $\bar{\opi}_h^{\text{d}}$, $\widetilde{G}_h$ the matrix for  $\nabla_h$, and $\widetilde{C}_h$ the matrix representing $\nabla_h^\perp$. Let  $\widetilde{A}_h=\widetilde{A}_h^\nabla$ be the matrix representation of $A_h^\nabla$, i.e., the surface nodal element stiffness matrix corresponding to the bilinear form $(\nabla_h\bullet,\nabla_h\bullet)_h+c(\bullet,\bullet)_h$ on $H_h(\nabla_h)$. By $\widetilde{\mathbf{A}}_h$  we denote the block diagonal matrix with $3$ copies of $\widetilde{A}_h$  as its block diagonal. In matrix notation, \eqref{kappabound}, \eqref{Bhcurldiv} translate into
\begin{equation}
\begin{aligned}
  &\widetilde{B}_h^{\nabla\times}=\big(\widetilde{D}^{\nabla\times}_h\big)^{-1}+\widetilde{P}_h^{\nabla\times}\widetilde{\mathbf{A}}_h^{-1}\big(\widetilde{P}_h^{\nabla\times}\big)^\top+c^{-1}\widetilde{G}_h\widetilde{A}_h^{-1}\widetilde{G}_h^\top,\\
&\widetilde{B}_h^{\nabla\cdot}=\big(\widetilde{D}^{\nabla\cdot}_h\big)^{-1}+\widetilde{P}_h^{\nabla\cdot}\widetilde{\mathbf{A}}_h^{-1}\big(\widetilde{P}_h^{\nabla\cdot}\big)^\top+c^{-1}\widetilde{C}_h\widetilde{A}_h^{-1}\widetilde{C}_h^\top,\\
&(\widetilde{A}_h^{\text{d}})^{-1}\lesssim\widetilde{B}_h^{\text{d}}\lesssim(\widetilde{A}_h^{\text{d}})^{-1},\quad\kappa(\widetilde{B}_h^{\text{d}}\widetilde{A}_h^{\text{d}})\lesssim1.
\end{aligned}
\end{equation}
Due to the condition number estimate given above,  PCG for \eqref{Vhd} preconditioned by $\widetilde{B}_h^\text{d}$ converges within uniformly bounded number of iterations (cf.~\cite{Xu1992}). 
In practice, the matrix inverses $\widetilde{\mathbf{A}}_h^{-1}$, $\widetilde{A}_h^{-1}$ could be approximated by any well-established fast Poisson solver on surfaces. For example, at the presence of a grid hierarchy, we are allowed to evaluate $\widetilde{\mathbf{A}}_h^{-1}$, $\widetilde{A}_h^{-1}$ using surface geometric multigrid methods in e.g., \cite{KornhuberYserentant2008,BonitoPasciak2012,Li2021SISC}. On unstructured triangulated surfaces, replacing $\widetilde{\mathbf{A}}_h^{-1}$, $\widetilde{A}_h^{-1}$ with AMG V- or W-cycle or BPX preconditioner in $\widetilde{B}_h^\text{d}$ still yields a quite efficient preconditioner.

\begin{remark}\label{remark3d}
The results in Sections \ref{secdeRham}-\ref{secHX} could be generalized to hypersurfaces without boundary. For a 3-dimensional  hypersurface $\mathcal{M}\subset\mathbb{R}^4$, we briefly explain preconditioners for the discrete problem \eqref{Vhd}. Given tangential vector fields $\bm{v}=(v_1,v_2,v_3,v_4)^\top$, $\bm{w}=(w_1,w_2,w_3,w_4)^\top$  along $\mathcal{M}$, we define the wedge product and 3-d surface curl  as
\begin{equation*}
    \bm{v}\wedge\bm{w}:=\left|
    \begin{array}{cccc}
    e_1&e_2&e_3&e_4\\
    v_1&v_2&v_3&v_4\\
    w_1&w_2&w_3&w_4\\
    \nu_1&\nu_2&\nu_3&\nu_4
    \end{array}
    \right|,\quad\nabla_\mathcal{M}\times\bm{v}:=\widetilde{\nabla}\wedge\bm{v},
\end{equation*}
where $\bm{\nu}=(\nu_1,\nu_2,\nu_3,\nu_4)^\top$ is the outward unit normal to $\mathcal{M}$, $\{e_i\}_{i=1}^4$ are the standard basis vectors in $\mathbb{R}^4,$ and $\widetilde{\nabla}$ is the gradient in $\mathbb{R}^4$. The de Rham complex on  $\mathcal{M}$ reads
\[\begin{CD}
    H(\nabla)@>\nabla_\mathcal{M}>>H(\nabla\times)@>\nabla_\mathcal{M}\times>>H(\nabla\cdot)@>\nabla_\mathcal{M}\cdot>>L^2(\mathcal{M}).
    \end{CD}
\]
The wedge product and 4-d curl of vector fields are given in \cite{GopNV2018} and are used for HX preconditioning on 4-d Euclidean regions.

We adopt the same notation used in Sections \ref{secdeRham}-\ref{secHX} with obvious generalized meanings in an ambient space $\mathbb{R}^4$ unless confusion arises. For example, $$\od^-=\nabla_\mathcal{M},~\od=\nabla_\mathcal{M}\times\quad\text{ or }\quad\od^-=\nabla_\mathcal{M}\times,~\od=\nabla_\mathcal{M}\cdot,$$ 
and $H_h(\od_h)$ in \eqref{Vhd} is the lowest-order 3-dimensional edge or face element space based on a triangulated hypersurface $\mathcal{M}_h$ with tetrahedral elements (cf.~\cite{Nedelec1980,Monk2003,ArnoldFalkWinther2009}).
In view of HX preconditioners on a 3-dimensional Euclidean region \cite{HiptmairXu2007} and the argument in Section \ref{secHX}, it is straightforward to derive preconditioners on a 3-d hypersurface $\mathcal{M}_h$
\begin{equation*}
\begin{aligned}
  &\widetilde{B}_h^{\nabla\times}=\big(\widetilde{D}^{\nabla\times}_h\big)^{-1}+\widetilde{P}_h^{\nabla\times}\widetilde{\mathbf{A}}_h^{-1}\big(\widetilde{P}_h^{\nabla\times}\big)^\top+c^{-1}\widetilde{G}_h\widetilde{A}_h^{-1}\widetilde{G}_h^\top,\\
&\widetilde{B}_h^{\nabla\cdot}=\big(\widetilde{D}^{\nabla\cdot}_h\big)^{-1}+\widetilde{P}_h^{\nabla\cdot}\widetilde{\mathbf{A}}_h^{-1}\big(\widetilde{P}_h^{\nabla\cdot}\big)^\top\\
  &\qquad+c^{-1}\widetilde{C}_h\big[\big(\widetilde{D}^{\nabla\times}_h\big)^{-1}+\widetilde{P}_h^{\nabla\times}\widetilde{\mathbf{A}}_h^{-1}\big(\widetilde{P}_h^{\nabla\times}\big)^\top\big]\widetilde{C}_h^\top,
\end{aligned}
\end{equation*}
where $\widetilde{C}_h$ is the matrix representing the 3-d discrete curl $\nabla_h\times$, and 
$\widetilde{\mathbf{A}}_h=\text{diag}(\widetilde{A}_h,\widetilde{A}_h,\widetilde{A}_h,\widetilde{A}_h)$ is a block diagonal matrix. It is possible to prove  $\widetilde{B}_h^{\nabla\times}$, $\widetilde{B}_h^{\nabla\cdot}$ are uniform preconditioners using the analysis in Sections \ref{secInterpPiola} and \ref{secHX} and tools in \cite{HiptmairXu2007,ABDG1998,Hiptmair2002}. We shall test the performance of $\widetilde{B}_h^{\nabla\times}$, $\widetilde{B}_h^{\nabla\cdot}$ in Section \ref{secNE}.
\end{remark}

\section{Computation of harmonic vector fields}\label{secharmonic}
In this section, we develop an iterative method for approximating the space of harmonic vector fields $\mathcal{H}(\od).$
For a tangential vector field $\bm{u}$ on $\mathcal{M},$ let $\sigma=-(\text{d}^-)^*\bm{u}.$ It is noted that $\bm{u}$ is harmonic if and only if $(\sigma,\bm{u})^\top$ satisfies
\begin{equation}\label{mixedHodge}
\begin{aligned}
(\sigma,\tau)_{\mathcal{M}}+(\text{d}^-\tau,\bm{u})_{\mathcal{M}}&=0,\quad\forall\tau\in H(\text{d}^-),\\
(\text{d}^-\sigma,\bm{v})_{\mathcal{M}}-(\text{d}\bm{u},\text{d}\bm{v})_{\mathcal{M}}&=0,\quad\forall \bm{v}\in H(\text{d}).
\end{aligned}
\end{equation}
In fact, $\mathcal{H}(\od)=N(\text{d}^-(\text{d}^-)^*+\text{d}^*\text{d})$ is the kernel of the Hodge Laplacian,
and \eqref{mixedHodge} is the mixed variational formulation of
$$(\text{d}^-(\text{d}^-)^*+\text{d}^*\text{d})\bm{u}=0.$$
In the discrete level, we consider the space of discrete harmonic vector fields
\[
\mathcal{H}_h(\text{d}_h):=\big\{\bm{v}_h\in H_h(\text{d}_h): \text{d}_h\bm{v}_h=0,~(\text{d}^-_h\tau_h,\bm{v}_h)_h=0~\forall\tau_h\in H_h(\text{d}_h^-)\big\}.
\]
Let $(\text{d}_h^-)^*$ be the $L^2(\mathcal{M}_h)$-adjoint of  $\text{d}_h^-: H_h(\od_h^-)\rightarrow H_h(\od_h)$ and $\sigma_h=-(\text{d}_h^-)^*\bm{u}_h$. Then $\bm{u}_h\in\mathcal{H}_h(\text{d}_h)$ if and only if $\sigma_h$ and $\bm{u}_h$ satisfy
\begin{equation}\label{discreteharmonic}
\begin{aligned}
(\sigma_h,\tau_h)_h+(\text{d}_h^-\tau_h,\bm{u}_h)_h&=0,\quad\forall\tau_h\in H_h(\text{d}_h^-),\\
(\text{d}_h^-\sigma_h,\bm{v}_h)_h-(\text{d}_h\bm{u}_h,\text{d}_h\bm{v}_h)_h&=0,\quad \forall \bm{v}_h\in H_h(\text{d}_h).
\end{aligned}
\end{equation}
Let $X_h:=H_h(\text{d}_h^-)\times H_h(\text{d}_h)$ and consider the discrete operator $\mathcal{A}_h: X_h\rightarrow X_h^\prime$
\[
\langle\mathcal{A}_h(\sigma_h,\bm{u}_h),(\tau_h,\bm{v}_h)\rangle:=(\sigma_h,\tau_h)_h+(\text{d}_h^-\tau_h,\bm{u}_h)_h+(\text{d}_h^-\sigma_h,\bm{v}_h)_h-(\text{d}_h\bm{u}_h,\text{d}_h\bm{v}_h)_h.
\]
In a compact block form, $\mathcal{A}_h$ reads
\begin{equation*}
\mathcal{A}_h=\begin{pmatrix}
I&~\od_h^-\\
\od_h^-&~-\od_h^*\od_h
\end{pmatrix}.
\end{equation*}
It is clear that 
\begin{equation}\label{HdhNAh}
  \mathcal{H}_h(\text{d}_h)=\big\{\bm{u}_h\in H_h(\text{d}_h): (-(\text{d}_h^-)^*\bm{u}_h,\bm{u}_h)^\top\in N(\mathcal{A}_h)\big\}.
\end{equation}
Therefore computing the discrete harmonic space 
is equivalent to finding a basis for the kernel of $\mathcal{A}_h.$ The dimensions of $\mathcal{H}_h(\nabla_h\times)$ and $\mathcal{H}_h(\nabla_h\cdot)$ are equal to the 1st Betti number of $\mathcal{M}_h$.  The operator $\mathcal{A}_h$ is singular when $\mathcal{M}_h$ has nontrivial cohomology groups. 

\subsection{MINRES for singular problems}
We shall construct a SPD preconditioner 
$\mathcal{B}_h: X^\prime_h\rightarrow X_h$ for $\mathcal{A}_h$ such that the condition number of $\mathcal{B}_h\mathcal{A}_h$ is uniformly bounded in certain sense even though $\mathcal{A}_h$ is singular. Let $\{\psi_i\}_{i=1}^K$ be a finite element basis of $X_h$,  $\{\psi_i^\prime\}_{i=1}^K$ the dual basis of $X_h^\prime$ such that $\langle\psi_i^\prime,\psi_j\rangle=\delta_{ij}$. Let $\widetilde{\mathcal{A}}_h$ and $\widetilde{\mathcal{B}}_h$ denote the matrix representations for $\mathcal{A}_h$ and $\mathcal{B}_h$ under these basis, respectively. We choose a random vector $b\in\mathbb{R}^K$ and the consider the algebraic system
\begin{equation}\label{minres}
    \widetilde{\mathcal{B}}_h\widetilde{\mathcal{A}}_hx=\widetilde{\mathcal{B}}_hb.
\end{equation}
In our case of interest, $\widetilde{\mathcal{A}}_h$ is rank-deficient and $b$ is almost surely not contained in the range of $\widetilde{\mathcal{A}}_h$. 
In other words, \eqref{minres} is not compatible provided $\widetilde{\mathcal{A}}_h$ has a nontrivial kernel. 
Nevertheless, the classical preconditioned MINRES method \cite{PaigeSaunders1975} minimizes the residual $\|b-\widetilde{\mathcal{A}}_hx\|_{\widetilde{\mathcal{B}}_h}$ and returns an iterative solution $x_k$ approximating the least-squares solution $x^\dagger$ for the singular problem \eqref{minres}, see \cite{ChoiPaigeSaunders2011,Choi2007}. Here $x^\dagger$ may not be the minimum length least-squares solution.

Due to the minimum residual or least-squares property 
\[
\|b-\widetilde{\mathcal{A}}_hx^\dagger\|_{\widetilde{\mathcal{B}}_h}=\min_{y\in\mathbb{R}^K}\|b-\widetilde{\mathcal{A}}_hy\|_{\widetilde{\mathcal{B}}_h},
\] 
we have
\[
(b-\widetilde{\mathcal{A}}_hx^\dagger,y)_{\widetilde{\mathcal{B}}_h}=0,\quad\forall y\in R(\widetilde{\mathcal{A}}_h),
\]
which implies that \begin{equation}\label{Brkernel}
    \widetilde{\mathcal{B}}_h(b-\widetilde{\mathcal{A}}_hx^\dagger)\in R(\widetilde{\mathcal{A}}_h)^\perp=N(\widetilde{\mathcal{A}}_h).
\end{equation}
Combining it with \eqref{HdhNAh}, we have that the $u_h$-part of the vector $\widetilde{\mathcal{B}}_h(b-\widetilde{\mathcal{A}}_hx^\dagger)$ represents a discrete tangential harmonic vector field on $\mathcal{M}_h$.

Let $\{\lambda_i\}_{i=1}^K$ be the set of eigenvalues of $\widetilde{\mathcal{B}}_h\widetilde{\mathcal{A}}_h$, arranged according to their absolute values in ascending order, that is,
\[
0=|\lambda_1|=\cdots=|\lambda_{m-1}|<|\lambda_m|\leq\cdots\leq|\lambda_K|.
\]
The next theorem shows that the convergence speed of MINRES for \eqref{minres} is determined by the effective condition number $$\hat{\kappa}(\widetilde{\mathcal{B}}_h\widetilde{\mathcal{A}}_h):=|\lambda_K/\lambda_m|.$$
\begin{theorem}\label{minresthm}
Let $b_R$ be the orthogonal projection of $b$ onto $R(\widetilde{\mathcal{A}}_h)$ with respect to the inner product $(\bullet,\bullet)_{\widetilde{\mathcal{B}}_h}$.
Let
$x_0$ be the initial guess, $x_k$ the MINRES iterative solution at the $k$-th step, and $r_k=b_R-\widetilde{\mathcal{A}}_hx_k$ for $k=0, 1, 2, \ldots$  Then we have
\[
\|r_k\|_{\widetilde{\mathcal{B}}_h}\leq2\left(\frac{\hat{\kappa}(\widetilde{\mathcal{B}}_h\widetilde{\mathcal{A}}_h)-1}{\hat{\kappa}(\widetilde{\mathcal{B}}_h\widetilde{\mathcal{A}}_h)+1}\right)^\frac{k}{2}\|r_0\|_{\widetilde{\mathcal{B}}_h}.
\]
\end{theorem}
\begin{proof}
Let $b_N$ be the orthogonal projection of $b$ onto $N(\widetilde{\mathcal{A}}_h)$ with respect to the inner product $(\bullet,\bullet)_{\widetilde{\mathcal{B}}_h}$. Without loss of generality, we assume $x_0=0.$ Let $\mathcal{K}(\widetilde{\mathcal{A}}_h,b,\ell):=\text{span}\{b,\widetilde{\mathcal{A}}_hb,\ldots,\widetilde{\mathcal{A}}^{\ell-1}_hb\}$ be the Krylov subspace. The property of MINRES  implies
\[
\|b-\widetilde{\mathcal{A}}_hx_k\|_{\widetilde{\mathcal{B}}_h}=\min_{y\in\mathcal{K}(\widetilde{\mathcal{A}}_h,b,k)}\|b-\widetilde{\mathcal{A}}_hy\|_{\widetilde{\mathcal{B}}_h}.
\] 
It then follows from 
$\|b_R-\widetilde{\mathcal{A}}_hx_k\|_{\widetilde{\mathcal{B}}_h}^2=\|b-\widetilde{\mathcal{A}}_hx_k\|^2_{\widetilde{\mathcal{B}}_h}-\|b_N\|^2_{\widetilde{\mathcal{B}}_h}$ that
\begin{align*}
    \|b_R-\widetilde{\mathcal{A}}_hx_k\|_{\widetilde{\mathcal{B}}_h}&=\min_{y\in\mathcal{K}(\widetilde{\mathcal{A}}_h,b,k)}\|b_R-\widetilde{\mathcal{A}}_hy\|_{\widetilde{\mathcal{B}}_h}\\
    &=\min_{y\in b_N+\mathcal{K}(\widetilde{\mathcal{A}}_h,b_R,k)}\|b_R-\widetilde{\mathcal{A}}_hy\|_{\widetilde{\mathcal{B}}_h},
\end{align*}
a minimum residual property of the consistent system $\widetilde{\mathcal{A}}_hx=b_R$.
Therefore applying the standard error analysis of Krylov subspace methods (cf.~\cite{Ma2016,Saad2003}) 
 to MINRES for the consistent system $\widetilde{\mathcal{A}}_hx=b_R$ yields
\begin{equation*}
\|r_k\|_{\widetilde{\mathcal{B}}_h}\leq2\left(\frac{\kappa\left(\widetilde{\mathcal{B}}_h\widetilde{\mathcal{A}}_h|_{R(\widetilde{\mathcal{A}}_h)}\right)-1}{\kappa\left(\widetilde{\mathcal{B}}_h\widetilde{\mathcal{A}}_h|_{R(\widetilde{\mathcal{A}}_h)}\right)+1}\right)^\frac{k}{2}\|r_0\|_{\widetilde{\mathcal{B}}_h}.
\end{equation*}
Due to $R(\widetilde{\mathcal{A}}_h)=N(\widetilde{\mathcal{A}}_h)^\perp,$ we have $\kappa\left(\widetilde{\mathcal{B}}_h\widetilde{\mathcal{A}}_h|_{R(\widetilde{\mathcal{A}}_h)}\right)=|\lambda_K/\lambda_m|=\hat{\kappa}(\widetilde{\mathcal{B}}_h\widetilde{\mathcal{A}}_h)$.  The proof is complete.
\qed\end{proof}
It follows from Theorem \ref{minresthm} that
$\|\widetilde{\mathcal{B}}_hr_k\|_{\widetilde{\mathcal{B}}_h^{-1}}=\|r_k\|_{\widetilde{\mathcal{B}}_h}\rightarrow0$ and \begin{equation*}
   \widetilde{\mathcal{B}}_h(b-\widetilde{\mathcal{A}}_hx_k)=\widetilde{\mathcal{B}}_h(b-b_R)+\widetilde{\mathcal{B}}_hr_k\xrightarrow{k\rightarrow\infty}\widetilde{\mathcal{B}}_h(b-b_R)\in N(\widetilde{\mathcal{A}}_h). 
\end{equation*}
When applying MINRES to \eqref{minres}, the stopping criterion could no longer be the norm of $b-\widetilde{\mathcal{A}}_hx_k$ because \eqref{minres} has no solution and $b-\widetilde{\mathcal{A}}_hx^\dagger\neq0$. In view of \eqref{Brkernel} and $\widetilde{\mathcal{A}}_h\widetilde{\mathcal{B}}_h(b-\widetilde{\mathcal{A}}_hx^\dagger)=0$, the norm of $\widetilde{\mathcal{A}}_h\widetilde{\mathcal{B}}_h(b-\widetilde{\mathcal{A}}_hx_k)$ is a viable stopping criterion, see \cite{ChoiPaigeSaunders2011}. When that quantity is reduced below the given error tolerance at step $k$, we accept $\widetilde{\mathcal{B}}_h(b-\widetilde{\mathcal{A}}_hx_k)$ as a null vector of $\widetilde{\mathcal{A}}_h$ and the $u_h$-part of $\widetilde{\mathcal{B}}_h(b-\widetilde{\mathcal{A}}_hx_k)$ as a representation of a discrete tangential harmonic field.
Besides the classical MINRES, other Krylov subspace methods for singular least-squares problems could be found in e.g., \cite{ChoiPaigeSaunders2011}.

\subsection{Block diagonal HX preconditioning}
Natural bounds of the extreme eigenvalues $\lambda_m$, $\lambda_K$ of  $\widetilde{\mathcal{B}}_h\widetilde{\mathcal{A}}_h$ are hidden in the analytical property of $\mathcal{B}_h\mathcal{A}_h$ with a carefully chosen preconditioner $\mathcal{B}_h$.
Following the preconditioning framework for saddle-point systems in \cite{LoghinWathen2004,MardalWinther2011}, we let $\mathcal{B}^{\text{ex}}_h: X^\prime_h\rightarrow X_h$ be the Riesz representation of $X^\prime_h$. In matrix notation, $\mathcal{B}^{\text{ex}}_h$ is a block operator \begin{equation*}
\mathcal{B}^{\text{ex}}_h:=\begin{pmatrix}
\big(A_h^{\text{d}^-}\big)^{-1}&O\\
O&\big(A_h^{\text{d}}\big)^{-1}
\end{pmatrix}.
\end{equation*}
First we note that $\mathcal{B}^{\text{ex}}_h\mathcal{A}_h: X_h\rightarrow X_h$ is bounded, i.e.,
\begin{align*}
    &\big(\mathcal{B}^{\text{ex}}_h\mathcal{A}_h(\sigma_h,\bm{u}_h),(\tau_h,\bm{v}_h)\big)_{X_h}=\langle\mathcal{A}_h(\sigma_h,\bm{u}_h),(\tau_h,\bm{v}_h)\rangle\\
&\leq\big(\|\sigma_h\|^2_{H(\text{d}^-_h)}+\|\bm{u}_h\|^2_{H(\text{d}_h)}\big)^\frac{1}{2}\big(\|\tau_h\|^2_{H(\text{d}^-_h)}+\|\bm{v}_h\|^2_{H(\text{d}_h)}\big)^\frac{1}{2}.
\end{align*}
When $\mathcal{B}_h=\mathcal{B}_h^{\rm ex}$, the maximum absolute eigenvalue is bounded by
\begin{equation}\label{lambdaJ}
   |\lambda_K|\leq1.
\end{equation}
To estimate the condition number $\hat{\kappa}(\widetilde{\mathcal{B}}^{\text{ex}}_h\widetilde{\mathcal{A}}_h)$, we need the next lemma. 
\begin{lemma}\label{infsup}
Let $V_h=N(\od_h)^\perp\oplus R(\od_h^-).$ For any $\hat{\bm{v}}_h\in V_h$ and $\hat{\tau}_h\in H_h(\od^-_h)$, there exist $\hat{\bm{u}}_h\in V_h$ and $\hat{\sigma}_h\in H_h(\od^-_h)$ such that
\begin{equation*}
\langle\mathcal{A}_h(\hat{\sigma}_h,\hat{\bm{u}}_h),(\hat{\tau}_h,\hat{\bm{v}}_h)\rangle\geq\beta\|(\hat{\sigma}_h,\hat{\bm{u}}_h)\|_{X_h}\|(\hat{\tau}_h,\hat{\bm{v}}_h)\|_{X_h},
\end{equation*}
where $\beta>0$ depends only on the discrete Poincar\'e constant $c_{h,P}$.
\end{lemma}
\begin{proof}
Consider the decomposition
\begin{equation*}
    \hat{\bm{v}}_h=\bm{w}_h+\text{d}^-_h\rho_h
\end{equation*}
with $\bm{w}_h\in N(\od_h)^\perp$ and $\rho_h\in N(\od_h^-)^\perp$. The inequality \eqref{discretePoincare} implies
\begin{subequations}
\begin{align}
    \|\bm{w}_h\|&\leq c_{h,P}\|\text{d}_h\bm{w}_h\|=c_{h,P}\|\text{d}_h\hat{\bm{v}}_h\|,\label{wh}\\
    \|\rho_h\|&\leq c_{h,P}\|\text{d}_h^-\rho_h\|.\label{rhoh}
\end{align}
\end{subequations}
Let $s=1/c_{h,P}^2$ and define
\begin{align*}
    \hat{\sigma}_h&:=\hat{\tau}_h+s\rho_h\in H_h(\od_h^-),\\ \hat{\bm{u}}_h&:=-\hat{\bm{v}}_h+\text{d}_h^-\hat{\tau}_h\in V_h.
\end{align*}
Then using \eqref{rhoh} and a mean value inequality we have
\begin{equation*}
    \begin{aligned}
    &\langle\mathcal{A}_h(\hat{\sigma_h},\hat{\bm{u}}_h),(\hat{\tau}_h,\hat{\bm{v}}_h)\rangle=\|\hat{\tau}_h\|^2+s(\rho_h,\hat{\tau}_h)_h-(\text{d}_h^-\hat{\tau}_h,\hat{\bm{v}}_h)_h+\|\text{d}_h^-\hat{\tau}_h\|^2\\
    &\quad+(\text{d}_h^-\hat{\tau}_h,\hat{\bm{v}}_h)_h+s(\text{d}_h^-\rho_h,\hat{\bm{v}}_h)_h+\|\text{d}_h\hat{\bm{v}}_h\|^2\\
    &\quad\geq\frac{1}{2}\|\hat{\tau}_h\|^2+\|\text{d}_h^-\hat{\tau}_h\|^2+s\|\text{d}_h^-\rho_h\|^2-\frac{s^2}{2}\|\rho_h\|^2+\|\text{d}_h\hat{\bm{v}}_h\|^2.
    \end{aligned}
\end{equation*}
Using the above estimate and  \eqref{wh}, \eqref{rhoh}
leads to 
\begin{equation}\label{Ahbound}
    \begin{aligned}
    &\langle\mathcal{A}_h(\hat{\sigma}_h,\hat{\bm{u}}_h),(\hat{\tau}_h,\hat{\bm{v}}_h)\rangle\\
    &\geq\frac{1}{2}\|\hat{\tau}_h\|^2+\|\text{d}_h^-\hat{\tau}_h\|^2+\frac{1}{2c_{h,P}^2}\|\text{d}_h^-\rho_h\|^2+\|\text{d}_h\hat{\bm{v}}_h\|^2\\
    &\gtrsim\|\hat{\tau}_h\|^2_{H(\text{d}_h^-)}+\|\hat{\bm{v}}_h\|^2_{H(\text{d}_h)}.
    \end{aligned}
\end{equation}
On the other hand, it follows from the definitions of $\hat{\sigma}_h, \hat{\bm{u}}_h$ that 
\begin{equation}\label{sigmahuhbound}
    \|\hat{\sigma}_h\|_{H(\text{d}_h^-)}+\|\hat{\bm{u}}_h\|_{H(\text{d}_h)}\lesssim\|\hat{\tau}_h\|_{H(\text{d}_h^-)}+\|\hat{\bm{v}}_h\|_{H(\text{d}_h)}.
\end{equation}
Combining \eqref{Ahbound} and \eqref{sigmahuhbound} completes the proof.
\qed\end{proof}

In finite element exterior calculus, the discrete Hodge decomposition reads
$$H_h(\od_h)=N(\od_h)^\perp\oplus R(\od_h^-)\oplus\mathcal{H}_h(\od_h).$$
Hence $V_h=\mathcal{H}_h(\od_h)^\perp$ is the orthogonal complement of $\mathcal{H}_h(\od_h)$ in $H_h(\od_h)$. In the case $\mathcal{H}_h(\od_h)=\{0\},$ Lemma \ref{infsup} is the inf-sup condition of the mixed method for the Hodge Laplacian proved in \cite{ArnoldFalkWinther2006,ArnoldFalkWinther2010}. In general, a direct consequence of Lemma \ref{infsup} is the following partial inf-sup condition modulo $\mathcal{H}_h(\od_h)$
\begin{equation}\label{infsup2}
    \inf_{0\neq(\hat{\sigma}_h,\hat{\bm{u}}_h)\in H_h(\od^-_h)\times V_h}\sup_{0\neq(\hat{\tau}_h,\hat{\bm{v}}_h)\in H_h(\od^-_h)\times V_h}\frac{\big(\mathcal{B}^{\text{ex}}_h\mathcal{A}_h(\hat{\sigma}_h,\hat{\bm{u}}_h),(\hat{\tau}_h,\hat{\bm{v}}_h)\big)_{X_h}}{\|(\hat{\sigma}_h,\hat{\bm{u}}_h)\|_{X_h}\|(\hat{\tau}_h,\hat{\bm{v}}_h)\|_{X_h}}\geq\beta.
\end{equation}
We next show that $\mathcal{B}^{\rm ex}_h$ is a uniform preconditioner for computing $N(\mathcal{A}_h)$.
\begin{theorem}\label{condex}
When $\mathcal{B}_h=\mathcal{B}^{\rm ex}_h$, we have
\begin{equation*}
    \hat{\kappa}(\widetilde{\mathcal{B}}^{\rm ex}_h\widetilde{\mathcal{A}}_h)=\hat{\kappa}(\mathcal{B}^{\rm ex}_h\mathcal{A}_h)\leq\beta^{-1}.
\end{equation*}
\end{theorem}
\begin{proof}
Let $0\neq(\sigma_h,\bm{u}_h)\in H_h(\od^-_h)\times H_h(\od_h)$ be an eigenfunction associated with the extreme eigenvalue $\lambda_m$, i.e.,  $\mathcal{B}^{\rm ex}_h\mathcal{A}_h(\sigma_h,\bm{u}_h)=\lambda_m(\sigma_h,\bm{u}_h)$. Let $\hat{\bm{u}}_h\in V_h$ be the orthogonal projection of $\bm{u}_h$ onto $V_h$ with respect to $(\bullet,\bullet)_{H(\od_h)}.$ Due to $(0,\bm{u}_h-\hat{\bm{u}}_h)\in N(\mathcal{A}_h)$, we obtain that
\begin{equation}\label{eigen}
    \big(\mathcal{B}^{\rm ex}_h\mathcal{A}_h(\sigma_h,\hat{\bm{u}}_h),(\hat{\tau}_h,\hat{\bm{v}}_h)\big)_{X_h}=\lambda_m\big((\sigma_h,\hat{\bm{u}}_h),(\hat{\tau}_h,\hat{\bm{v}}_h)\big)_{X_h}
\end{equation}
for all $\hat{\tau}_h\in H_h(\od_h^-)$ and $\hat{\bm{v}}_h\in V_h.$
It then follows from \eqref{eigen} and \eqref{infsup2} with $\hat{\sigma}_h=\sigma_h$ that
\begin{equation}\label{lambdam}
    |\lambda_m|\geq\beta.
\end{equation}
Combining \eqref{lambdam}  and \eqref{lambdaJ} finishes the proof.
\qed\end{proof}

In practice, we replace the diagonal block $\big(A_h^{\text{d}}\big)^{-1}$ in $\mathcal{B}^{\text{ex}}_h$ with the surface HX preconditioners proposed in Section \ref{secHX} and obtain 
\begin{equation}\label{blockHX}
\mathcal{B}^{\text{HX}}_h:=\begin{pmatrix}
(A_h^{\nabla})^{-1}&O\\
O&B_h^{\text{d}}\end{pmatrix}.
\end{equation}
Using Theorem \ref{condex} and \eqref{kappabound}, we obtain the condition number estimate
\[
\hat{\kappa}(\widetilde{\mathcal{B}}^{\text{HX}}_h\widetilde{\mathcal{A}}_h)=\hat{\kappa}(\mathcal{B}^{\text{HX}}_h\mathcal{A}_h)\leq\beta^{-1}\kappa(B^\text{d}_hA^\text{d}_h)\lesssim1.
\]
It then follows from the above estimate and Theorem \ref{minresthm} that MINRES for \eqref{minres} with $\widetilde{\mathcal{B}}_h=\widetilde{\mathcal{B}}^{\text{HX}}_h$ uniformly converges with respect to the mesh size $h$.

\begin{figure}[tbhp]
\centering
\subfloat[]{\label{torus1}\includegraphics[width=5.5cm,height=5cm]{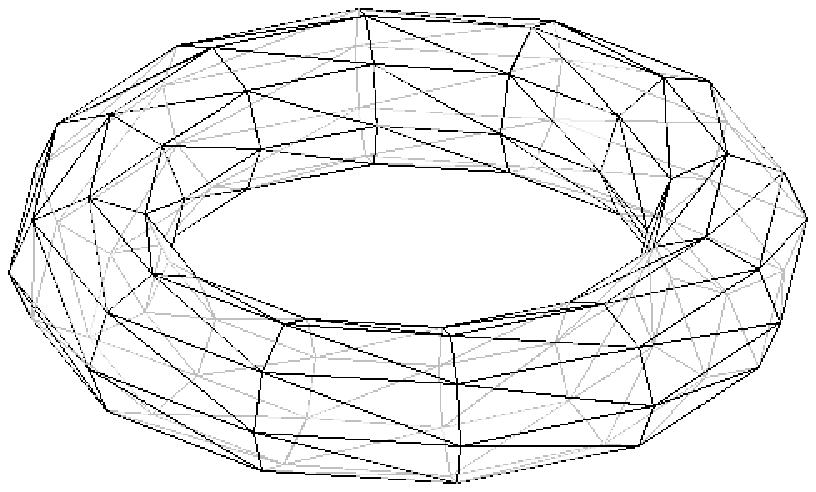}}
\subfloat[]{\label{torus2}\includegraphics[width=5.5cm,height=5cm]{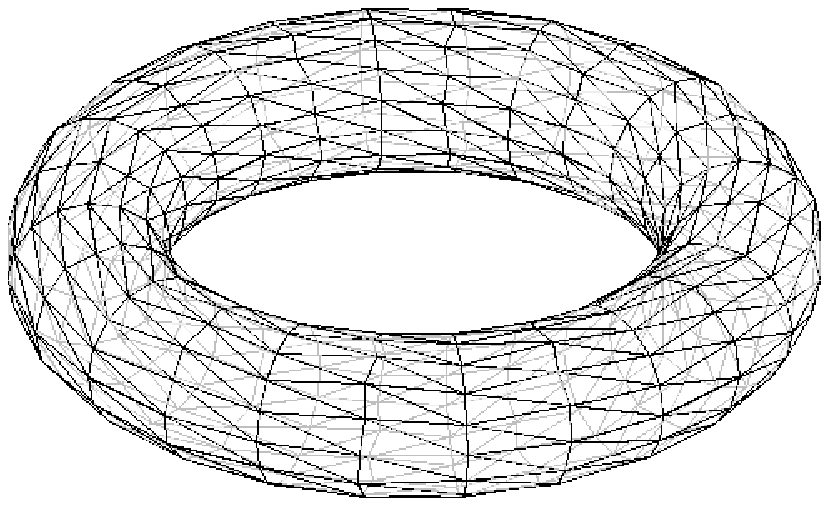}}
\caption{(a) An initial mesh $\mathbb{T}_0^2$, 192 elements; (b) A mesh with vertices on $\mathbb{T}^2$, 768 elements.}
\label{2tori}
\end{figure}

\begin{table}[ht]
\caption{PCG iterations for the $N_0$ and $RT_0$ element on $\mathbb{T}^2$ with $c=1$.}
\centering
\begin{tabular}{|c|c|c|c|c|c|}
\hline
$N$ & $B_{N_0}$&$E_{N_0}$ &$B_{RT_0}$ & $E_{RT_0}$ \\
\hline
192&15&3.090e-7&24&9.164e-7\\
768&16&8.952e-7&27&5.092e-7\\
3072&17&3.211e-7&28&3.176e-7\\
12288&17&3.684e-7&28&5.586e-7\\
49152&16&9.776e-7&28&8.382e-7\\
196608&16&5.936e-7&29&4.350e-7\\
786432&16&4.447e-7&29&5.264e-7\\
\hline
\end{tabular}
\label{T2N0RT0tau1}
\end{table}

\begin{table}[ht]
\caption{PCG iterations for the $N_0$ and $RT_0$ element on $\mathbb{T}^2$ with $c=10000$.}
\centering
\begin{tabular}{|c|c|c|c|c|c|}
\hline
$N$ & $B_{N_0}$&$E_{N_0}$ &$B_{RT_0}$ & $E_{RT_0}$ \\
\hline
192&15&6.216e-7&21&6.298e-7\\
768&21&4.509e-7&22&7.026e-7\\
3072&22&6.942e-7&23&6.672e-7\\
12288&21&6.597e-7&21&7.212e-7\\
49152&18&8.559e-7&18&8.445e-7\\
196608&14&7.024e-7&14&7.029e-7\\
786432&10&8.348e-7&10&8.304e-7\\
\hline
\end{tabular}
\label{T2N0RT0tau10000}
\end{table}

\begin{table}[ht]
\caption{Convergence history of discretization errors with $c=1$ on $\mathbb{S}^3$.}
\centering
\begin{tabular}{|c|c|c|c|c|c|c|}
\hline
$N$ & $\|\bm{u}-\bm{u}_h({N_0})\|$& order &$\|\bm{u}-\bm{u}_h({RT_0})\|$& order \\
\hline
128&2.056& &2.575&\\
1024&1.229&0.742 &1.176 &1.131\\
8192&6.460e-1&0.928 &5.097e-1&1.206\\
65536&3.280e-1& 0.978 &2.393e-1&1.091\\
524288&1.648e-1& 0.993 &1.175e-1&1.026\\
4194304&8.248e-2&0.996 &5.845e-2&1.007\\
\hline
\end{tabular}
\label{S3convergence}
\end{table}

\begin{table}[ht]
\caption{PCG iterations for the $N_0$ and $RT_0$ element on $\mathbb{S}^3$ with $c=1$.}
\centering
\begin{tabular}{|c|c|c|c|c|c|c|}
\hline
$N$ & $B_{N_0}$& $E_{N_0}$ & $B_{RT_0}$ & $E_{RT_0}$ \\
\hline
128&10&3.227e-7&15&5.754e-7\\
1024&11&6.385e-7&18 &3.382e-7\\
8192&12&4.564e-7 &18&8.321e-7\\
65536&15&8.714e-7&21&9.625e-7\\
524288&20&8.735e-7 &28&8.996e-7\\
4194304&27&7.356e-7 &36&8.731e-7\\
\hline
\end{tabular}
\label{S3N0RT0tau1}
\end{table}

\begin{table}[ht]
\caption{PCG iterations for the $N_0$ and $RT_0$ element on $\mathbb{S}^3$ with $c=10000$.}
\centering
\begin{tabular}{|c|c|c|c|c|c|c|}
\hline
$N$ & $B_{N_0}$& $E_{N_0}$ & $B_{RT_0}$ & $E_{RT_0}$ \\
\hline
128&14&9.790e-7&15&6.657e-7\\
1024&16&8.223e-7&16 &7.861e-7\\
8192&17&6.976e-7&16&7.470e-7\\
65536&16&9.711e-7&15&8.524e-7\\
524288&14&8.470e-7&14&5.353e-7\\
4194304&12&6.780e-7&12&3.922e-7\\
\hline
\end{tabular}
\label{S3N0RT0tau10000}
\end{table}

\begin{figure}[tbhp]
\centering
\subfloat[]{\label{harmonic1}\includegraphics[width=5.5cm,height=5cm]{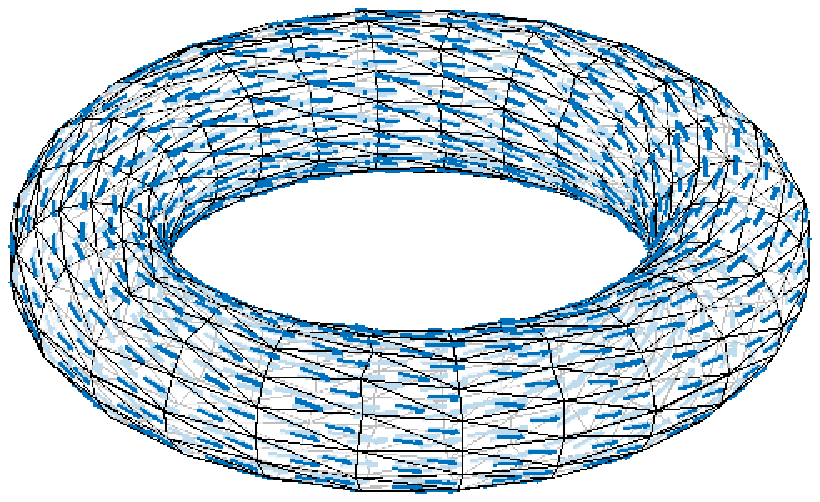}}
\subfloat[]{\label{harmonic2}\includegraphics[width=5.5cm,height=5cm]{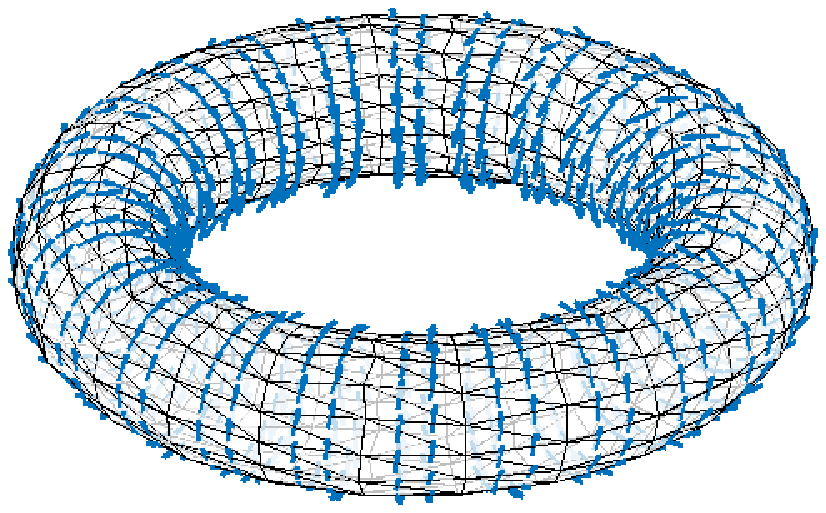}}
\caption{(a) Harmonic vector field I; (b) Harmonic vector field II.}
\label{harmonicfigure}
\end{figure}

\begin{table}[tbhp]
\caption{MINRES iterations for computing harmonic vector fields on $\mathbb{T}^2$.}
\centering
\begin{tabular}{|c|c|c|}
\hline
$N$ & $B_{P_1N_0}$ & $E_{P_1N_0}$ \\
\hline
192&44&5.388e-7\\
768&46&6.210e-7\\
3072&47&8.589e-7\\
12288&46&9.159e-7\\
49152&46&6.378e-7\\
196608&45&8.330e-7\\
786432&45&6.461e-7\\
\hline
\end{tabular}
\label{T2harmonic}
\end{table}

\section{Numerical experiments}\label{secNE}
This section is devoted to test the performance of the surface HX preconditioners for the lowest-order edge $N_0$ element and face $RT_0$ element on 2- and 3-dimensional hypersurfaces. In particular, we set $\mathcal{M}$ to be the 2-d torus
\begin{equation*}
    \mathbb{T}^2=\left\{x\in\mathbb{R}^3: \delta_{\mathbb{T}^2}(x):=\sqrt{\big((x_1^2+x_2^2)^\frac{1}{2}-R\big)^2+x_3^2}-r=0\right\}
\end{equation*}
with $R=2, r=0.5,$ 
and the unit 3-d sphere \begin{equation*}
\mathbb{S}^3=\big\{x\in\mathbb{R}^4: \delta_{\mathbb{S}^3}(x):=\big(x_1^2+x_2^2+x_3^2+x_4^2\big)^\frac{1}{2}-1=0\big\}.
\end{equation*}
We remark that the signed distance function $\delta=\delta_{\mathbb{T}^2}$ or $\delta_{\mathbb{S}^3}$ is used for refining meshes and is not required in the implementation of preconditioners.
In each table, let  $B_{N_0}$ (resp.~$B_{RT_0}$) denote the surface HX preconditioner $\widetilde{B}_h^{\nabla\times}$ (resp.~$\widetilde{B}_h^{\nabla\cdot}$) for \eqref{Vhd}. Let $\bm{u}_h({N_0})$ (resp.~$\bm{u}_h({RT_0})$)  denote the solution for \eqref{Vhd} based on the edge (resp.~face) element.  By $N$ we denote the number of grid elements. The iterative error of PCG method based on $B_{N_0}$ or $B_{RT_0}$ is denoted by $E_{N_0}$ or $E_{RT_0}$, respectively.

Here we explain the basis used in numerical implementation. Let $\{z_i\}_i$ be the set of grid vertices on $\mathcal{M}_h$, and $\lambda_i$ the continuous and piecewise linear hat function at $z_i$. We use 
 $$\big\{\lambda_i\od_h^-\lambda_j-\lambda_j\od_h^-\lambda_i\big\}_{z_i, z_j\text{ form an edge}}$$
 as a basis for $H_h(\nabla_h\times)$ (resp.~$H_h(\nabla_h\cdot)$ with $\dim\mathcal{M}_h=2$) where $\od^-_h=\nabla_h$ (resp.~$\od^-_h=\nabla_h^\perp$).
When $\dim\mathcal{M}_h=3$, a basis for $H_h(\nabla_h\cdot)$  is $$\big\{\lambda_i\nabla_h\lambda_j\wedge_h\nabla_h\lambda_k+\lambda_j\nabla_h\lambda_k\wedge_h\nabla_h\lambda_i+\lambda_k\nabla_h\lambda_i\wedge_h\nabla_h\lambda_j\big\}_{z_i, z_j, z_k\text{form a 2-d face}},$$ 
where $\wedge_h$ is the wedge product on $\mathcal{M}_h$ (see $\wedge$ on $\mathcal{M}$ in Remark \ref{remark3d}).

\subsection{Preconditioning on a 2-d torus}\label{exp1}
In this example, we consider the problem \eqref{Vd} with $f$ being the tangential component of the constant vector field $(1,1,1)$ on $\mathcal{M}=\mathbb{T}^2.$ The initial triangulation of $\mathbb{T}^2$ is $\mathbb{T}_0^2$ shown in Figure \ref{torus1}. The initial surface $\mathbb{T}_0^2$ is uniformly quad-refined (dividing each triangle into four subtriangles by connecting midpoints of all edges) to obtain a sequence of meshes on $\mathbb{T}_0^2$. Then the actual triangulated surface $\mathcal{M}_h$ is constructed by mapping grid vertices of meshes on $\mathbb{T}_0^2$ to $\mathbb{T}^2$ via $a$, see Figure \ref{torus2}.

To solve the SPD systems \eqref{Vhd}, we run the MATLAB function \textsf{pcg} with preconditioners $B_{N_0}$ and $B_{RT_0}$, where discrete Laplacians used in $B_{N_0}$ and $B_{RT_0}$ are solved by the operation `$\backslash$'. The stopping criterion for {\sf pcg} is $|\widetilde{B}_h^{\text{d}}r_k|/|b|\leq$1e-6, where $r_k$ is the PCG residual at the $k$-th step and $b$ is the right hand side of the algebraic system. 

It is observed from Tables \ref{T2N0RT0tau1} and \ref{T2N0RT0tau10000} that the HX preconditioners lead to uniformly convergent PCG method on $\mathbb{T}^2$. In addition, the number of PCG iterations is independent of the magnitude of $c\gg1.$

\subsection{Preconditioning on a 3-d sphere}

Let $p_1=(1, 0, 0, 0)$, $p_2=(0, 1, 0, 0)$, $p_3=(-1, 0, 0, 0)$, $p_4=(0, -1, 0, 0)$, 
$p_5=(0, 0, 1, 0)$, $p_6=(0, 0, -1, 0)$, $p_7=(0, 0, 0, 1)$, $p_8=(0, 0, 0, -1)$, and $[p_ip_jp_kp_\ell]$ denote the simplex in $\mathbb{R}^4$ with vertices $p_i, p_j, p_k, p_\ell$. The initial surface  $\mathbb{S}_0^3$ consists of the following 3-dimensional simplexes $[p_1p_2p_5p_7]$, $[p_3p_5p_2p_7]$, $[p_3p_4p_5p_7]$, $[p_1p_5p_4p_7]$, $[p_1p_6p_2p_7]$, $[p_3p_2p_6p_7]$, $[p_3p_6p_4p_7]$, $[p_1p_4p_6p_7]$, $[p_8p_1p_2p_5]$, $[p_8p_3p_5p_2]$, \\ $[p_8p_3p_4p_5]$, $[p_8p_1p_5p_4]$, $[p_8p_1p_6p_2]$, $[p_8p_3p_2p_6]$, $[p_8p_3p_6p_4]$, $[p_8p_1p_4p_6]$ in $\mathbb{R}^4$.

The initial surface $\mathbb{S}_0^3$ is uniformly refined by the red-refinement algorithm in \cite{Bey2000} to generate a grid sequence on $\mathbb{S}_0^3$. We use $a$ to map the grid vertices of refinement of $\mathbb{S}_0^3$  to construct the true triangluation $\mathcal{M}_h$. Let $\varphi=x_1+x_2+x_3+x_4$. We use $\bm{u}=\nabla_{\mathbb{S}^3}\varphi$ as the exact solution of \eqref{Vd} with $\text{d}=\nabla\times$ and $\text{d}=\nabla\cdot$. In Table \ref{S3convergence}, we record the discretization error of  \eqref{Vhd} with $\text{d}_h=\nabla_h\times$ and $\text{d}_h=\nabla_h\cdot$, which clearly exhibits first-order convergence.

The discrete problem in \eqref{Vhd} is solved by the MATLAB function \textsf{pcg} with the same setup in Subsection \ref{exp1}.
We use the classical AMG V-cycle in the iFEM package \cite{iFEM} as the discrete Poisson solver in $B_{N_0}$ and $B_{RT_0}$. 

It can be observed from Tables \ref{S3N0RT0tau1} and \ref{S3N0RT0tau10000} that the HX-preconditioned PCG method uniformly converges on $\mathbb{S}^3$. Moreover, the convergence rate of MINRES iteration is robust with respect to the large parameter $c$. 

\subsection{Harmonic vector fields on a 2-d torus}

In the third experiment, we compute the space of harmonic vector fields $\mathcal{H}_h(\nabla_h\times)$ on a triangulated torus. The torus, initial mesh, and mesh refinement are the same as Subsection \ref{exp1}. The HX-preconditioned MINRES method in Section \ref{secharmonic} is applied to solve the kernel of the system \eqref{discreteharmonic} with $\text{d}_h^-=\nabla_h$, $\text{d}_h=\nabla_h\times$.  The right hand side $b$ in \eqref{minres} is randomly produced by the MATLAB function \textsf{rand}. The preconditioner $\mathcal{B}_h^{\text{HX}}=B_{P_1N_0}$ is given in \eqref{blockHX} with  $\text{d}=\nabla\times$, where  all discrete Laplacians $A_h^\nabla$ used in  $B_{P_1N_0}$ are inverted by `$\backslash$'. The stopping criterion is 
\[
\frac{|\widetilde{\mathcal{A}}_h\widetilde{\mathcal{B}}^{\text{HX}}_h(b-\widetilde{\mathcal{A}}_hx_k)|}{|b|}\leq10^{-6}.
\]
In Table \ref{T2harmonic}, $E_{P_1N_0}$ is the iterative error of MINRES preconditioned by $B_{P_1N_0}$.

On discrete tori in this experiment, the dimension of $\mathcal{H}_h(\nabla_h\times)$ is 2. We use MINRES to solve \eqref{minres} twice with two different randomly generated $b$. It is shown in Table \ref{T2harmonic} that the number of MINRES iterations is uniformly bounded.  The Gram--Schmidt process is applied to the two output vector fields from MINRES with respect to the $H_h(\nabla_h\times)$-norm. The resulting two orthonormal harmonic vector fields are shown in Figure \ref{harmonicfigure}. 


\end{document}